\documentclass[10pt]{article}
\usepackage[all,2cell]{xy} \UseAllTwocells \SilentMatrices
\usepackage{latexsym,amsfonts,amssymb}
\usepackage{amsmath,amsthm,amscd}
\usepackage{hyperref,psfrag}
\usepackage{diagcat}
\usepackage{color}
\usepackage{etoolbox}
\usepackage[dvips]{epsfig}
\usepackage{psfrag}
\usepackage{youngtab}

\usepackage{graphicx}
\usepackage{a4wide}
\usepackage{geometry}

\usepackage{epigraph,wrapfig}

\normalfont\upshape

\usepackage{fancyhdr}
\theoremstyle{definition}
\newtheorem{thm}{Theorem}[section]
\newtheorem{cor}[thm]{Corollary}

\newtheorem{lem}[thm]{Lemma}
\newtheorem{rem}[thm]{Remark}
\newtheorem{prop}[thm]{Proposition}
\newtheorem{defn}[thm]{Definition}

\newtheorem*{thm*}{Theorem}
\numberwithin{equation}{section}

\usepackage{bbm}

\def\N{{\mathbbm N}}
\def\R{{\mathbbm R}}
\def\Z{{\mathbbm Z}}

\def\F{{\mathbbm F}}

\def\1{{\mathbbm{1}}}

\def\dif{\partial}

\def\lra{{\longrightarrow}}
   
\def\dmod{{\mathrm{\mbox{-}mod}}}   

\def\Id{\mathrm{Id}}
\def\mc{\mathcal}
\def\mf{\mathfrak}
\def\shuffle{\,\raise 1pt\hbox{$\scriptscriptstyle\cup{\mskip
               -4mu}\cup$}\,}

\newcommand{\refequal}[1]{\xy {\ar@{=}^{#1}
(-1,0)*{};(1,0)*{}};
\endxy}

\newcommand{\mH}{\mathrm{H}} 
\newcommand{\mHH}{\mathrm{HH}}
\newcommand{\mHHH}{\mathrm{HHH}}
\newcommand{\mtHHH}{\widehat{\mathrm{HHH}}}
\newcommand{\pH}{p\mathrm{H}}
\newcommand{\pHH}{p\mathrm{HH}}
\newcommand{\pHHH}{p\mathrm{HHH}}
\newcommand{\ptHHH}{p\widehat{\mathrm{HHH}}}
\newcommand{\pC}{pC}
\newcommand{\pT}{pT}

\newstrandstyle{Thick}{type=ribbon,style=solid,ribboncolor=black,color=black}

\title{On some $p$-differential graded link homologies II}
\author{You Qi and Joshua Sussan}

\date{August 24, 2021}

\begin{document}
%

\maketitle

\begin{abstract}
In \cite{QiSussanLink}, a link invariant categorifying the Jones polynomial at a $2p$th root of unity, where $p$ is an odd prime, was constructed.  This categorification utilized an $N=2$ specialization of a differential introduced by Cautis. Here we give a family of link homologies where the Cautis differential is specialized to a positive integer of the form $N=kp+2$.  When $k$ is even, all these link homologies categorify the Jones polynomial evaluated at a $2p$th root of unity, but they are non-isomorphic invariants.

\end{abstract}

\setcounter{tocdepth}{2} \tableofcontents

\section{Introduction}
Given any link $L$, Khovanov and Rozansky constructed a triply-graded link homology theory $\mHHH(L)$ whose graded Euler characteristic is the HOMFLYPT polynomial of $L$ \cite{KR2} using the theory of matrix factorizations.  Khovanov reformulated this construction using categories of Soergel bimodules \cite{KR3}. The connection between Soergel bimodules and link homology began with Rouquier's categorification of the braid group \cite{RouBraid2}.  He also extended this categorification to a link homology \cite{Roulink}.  In a later work \cite{KRWitt}, Khovanov and Rozansky equipped this link homology with an action of the positive half of the Witt algebra.  

Cautis defined an additional differential, depending upon a natural number $N$, on the chain groups for the triply-graded theory, which produced a categorification of the quantum $\mathfrak{sl}_N$-link invariant.  Independently, Robert and Wagner \cite{RW} and Queffelec, Rose, and Sartori \cite{QRS} constructed the same $\mathfrak{sl}_N$-link homology from different perspectives.

Link homology theories are important examples of categorification. In 1994, Crane and Frenkel \cite{CF} introduced their categorification program with the purpose of constructing $(3+1)$-dimensional TQFTs by lifting the $(2+1)$-dimensional TQFTs coming from quantum groups.  The $(2+1)$-dimensional TQFTs utilize quantum groups at roots of unity.  Motivated by this goal, Khovanov introduced the subject of hopfological algebra \cite{Hopforoots}, which was further developed in \cite{QYHopf}.  The basic idea is to take a categorification of a quantum group (for a generic quantum parameter) or its representations, defined over a field of characteristic $p$ and look for differentials $\partial$ such that $\partial^p=0$.  Searching for such $p$-differentials is equivalent to constructing an action of the Hopf algebra $H=\Bbbk[\dif]/(\dif^p)$. We refer the reader to \cite{QiSussan2} for a survey of some recent progress in this direction.

One of the Witt algebra generators (denoted $L_1=x^2\frac{\partial}{\partial x}$) in \cite{KRWitt}  acts as a $p$-differential over a field of characteristic $p$ on $\mHHH(L)$. For degree reasons, this is the only Witt algebra generator that can play the role of a $p$-differential. In \cite{QiSussanLink}, we utilized this $p$-differential along with the Cautis differential for $N=2$, to construct a categorification of the Jones polynomial evaluated at a $2p$th root of unity. The Cautis differential has the effect of wedging on $\mHHH(L)$ with $L_1$. A key property that facilitated the construction in \cite{QiSussanLink} is that the two actions of $L_1$, as the $p$-differential and the Cautis differential, commute with each other.

In this work, we generalize the previous results by considering the Cautis differential for $N=kp+2$ where $p$ is an odd prime. The essential reason that this generalization works is that, in characteristic $p$, the polynomial algebra generated by $x^p$ lies in the center of the Witt algebra. Therefore the $p$-differential $L_1$ still commutes with $L_{kp+1}=x^{kp+2}\frac{\partial}{\partial x}$, the latter now serving as the Cautis differential. Thus, for each $N=kp+2$ and braid $\beta$, we obtain an object $p{\mH}(\beta,kp+2)$ in the homotopy category of $p$-complexes.  Our main result is the following.

\begin{thm*}
Let $L$ be a link presented as the closure of a braid $\beta$ and $p$ be an odd prime.
The object $p{\mH}(\beta,kp+2)$ is a finite-dimensional framed link invariant.  When $k \in 2 \Z$, its Euler characteristic is the Jones polynomial evaluated at a $2p$th root of unity.
\end{thm*}

Varying the Cautis differential categorifies $\mathfrak{sl}_N$-link invariants for different ranks.  But when $q$ is a $2p$th root of unity, and $k$ is even, $q^{kp+2}=q^2$ so the $\mathfrak{sl}_{kp+2}$-link invariant is just the Jones polynomial. While this is true on the decategorified level, we show in Section \ref{secexamples} that on the level of homology, the invariant for the Hopf link depends upon $k$.
Thus we obtain a family of distinct link homologies categorifying the Jones polynomial at $2p$th roots of unity.

In a parallel direction \cite{QRSW}, we will show that the root-of-unity categorification of \cite{QiSussanLink} can be extended to the colored case. Combining the approach of \cite{QRSW} with the current work, one can construct certain colored $\mf{sl}_N$-link homologies, which we plan to explore.

\paragraph{Acknowledgements.}The authors would like to thank Louis-Hadrien Robert and Emmanuel Wagner for helpful conversations.

While working on the project, Y.~Q.~was partially supported by the NSF grant DMS-1947532.
J.~S.~is partially supported by the NSF grant DMS-1807161 and PSC CUNY Award 63047-0051.

\section{Background} \label{secbackground}
In this section, we recall some necessary background material from \cite{QiSussanLink}.

\subsection{\texorpdfstring{$p$}{p}-DG algebras and their relative homotopy categories}
Let $\Bbbk$ be a field of characteristic $p>2$.
For any graded or ungraded algebra $B$ over $\Bbbk$, denote by $d_0$ the zero super differential and by $\dif_0$ the zero $p$-differential on $B$, while letting $B$ sit in homological degree zero. When $B$ is graded, the homological grading is independent of the internal grading of $B$.

We will let $\mc{C}(B,d_0)$ and $\mc{C}(B,\dif_0)$ stand for the use homotopy categories and $p$-homotopy categories of $B$ respectively. For more details on hopfological algebra of $p$-homotopy categories, see \cite{Hopforoots, QYHopf}.

For a graded module $M$ over a graded algebra $B$, we let $M \{n\}$ denote the module $M$, where the internal grading has been shifted up by $n$.  When convenient, we sometimes call this shifted module $q^n M$.

We will need the following functor introduced in \cite[Section 2.1]{QiSussanLink}. Let $B$ be a $\Bbbk$-algebra.
Given a chain complex of $B$-modules, we repeat every term sitting in odd homological degrees $(p-1)$ times while keeping even degree terms unchanged. More explicitly, for a given complex
\[
\xymatrix{
\cdots \ar[r]^-{d_{2k+2}}  & M_{2k+1} \ar[r]^-{d_{2k+1}} & M_{2k} \ar[r]^-{d_{2k}} & M_{2k-1} \ar[r]^-{d_{2k-1}} & M_{2k-2} \ar[r]^-{d_{2k-2}} & \cdots
} ,
\]
the $p$-extended complex looks like
\[
\xymatrix{\cdots \ar[r]^-{d_{2k+2}} & M_{2k+1} \ar@{=}[r]& \cdots \ar@{=}[r] & M_{2k+1}\ar[r]^-{d_{2k+1}} \ar[r] &
M_{2k}\ar `r[rd] `_l `[llld] _-{d_{2k}} `[d] [lld]
& \\
& & M_{2k-1}\ar@{=}[r]&\cdots \ar@{=}[r]& M_{2k-1} \ar[r]^-{d_{2k-1}}& M_{2k-2} \ar[r]^{d_{2k-2}}& \cdots}
\ .
\]
Similarly, for chain maps maps of $B$-modules, the odd degree maps are repeated $p-1$ times while the even ones are kept unchanged. In \cite[Proposition 2.3]{QiSussan}, it is shown that this construction leads to an exact functor between homotopy categories
\begin{equation}\label{eqn-p-extension}
    \mc{P}: \mc{C}(B,d_0) \lra \mc{C}(B,\dif_0).
\end{equation}
This will be referred to as the \emph{$p$-extension functor}. The exactness of $\mc{P}$ means that it commutes with homological shifts, denoted $[\pm 1]_d$ and $[\pm 1]_\dif$ respectively, on $\mc{C}(B,d_0)$ and $\mc{C}(B,\dif_0)$, and preserves the class of distinguished triangles.

Suppose $(A,\dif_A)$ is a $p$-DG algebra, i.e., a graded algebra equipped with a differential $\dif_A$ of degree two, satisfying
\begin{equation}
\dif_A^p(a)\equiv 0 , \quad \quad \dif_A(ab)=\dif_A(a)b+a\dif_A(b),
\end{equation}
for all $a,b\in A$. In other words, $A$ is an algebra object in the module category of the graded Hopf algebra $H_q=\Bbbk[\dif_q]/(\dif_q^p)$, where the primitive degree-two generator $\dif_q\in H_q$ acts on $A$ by the differential $\dif_A$. Below we will usually take $B$ to be a certain smash product algebra associated with $(A,\dif_A)$, which we next recall.

Given a $p$-DG algebra $A$, we may form the \emph{smash product algebra} $A\# H_q$ in this case. As a $\Bbbk$-vector space, $A\# H_q$ is isomorphic to $A\otimes H_q$, subject to the multiplication rule determined by 
\begin{equation}
(a\otimes \dif_q)(b\otimes \dif_q)=ab\otimes \dif_q^2+ a\dif_A(b)\otimes \dif_q.
\end{equation}
Notice that, by construction, $A\otimes 1$ and $1\otimes H_q$ sit in $A\# H_q$ as subalgebras.

For later use, let us record a family of \emph{balanced} $H_q$-modules
\begin{equation}\label{eqn-Vi}
    V_i:=
    \left(
    \xymatrix{
\overset{-i}{\Bbbk} \ar@{=}[r] & \overset{-i+2}{\Bbbk} \ar@{=}[r] & \cdots \ar@{=}[r] & \overset{i-2}{\Bbbk} \ar@{=}[r] & \overset{i}{\Bbbk}
}
    \right)
\end{equation}
for each $i$ in $\{ 0, \dots, p-1 \}$.  As graded modules over $H_q$, we have $V_i\cong q^{-i}H_q/(\dif_q^{i+1})$.

We will also need a relative version of certain homotopy categories that played an essential role in \cite{QiSussanLink}. There is an exact forgetful functor between the usual homotopy categories of chain complexes of graded $A\# H_q$-modules 
\[
\mc{F}_d: \mc{C}(A\# H_q,d_0)\lra \mc{C}(A,d_0).
\]
An object $K_\bullet$ in $\mc{C}(A\# H_q,d_0)$ lies inside the kernel of the functor if and only if, when forgetting the $H_q$-module structure on each term of $K_\bullet$, the complex of graded $A$-modules $\mc{F}_d(K_\bullet)$ is null-homotopic. The null-homotopy map on $\mc{F}_d(K_\bullet)$, though, is not required to intertwine $H_q$-actions.
 
Likewise, there is an exact forgetful functor 
\[
\mc{F}_\dif: \mc{C}(A\# H_q,\dif_0)\lra \mc{C}(A,\dif_0).
\]
Similarly, an object $K_\bullet$ in $\mc{C}(A\# H_q,\dif_0)$ lies inside the kernel of the functor if and only if, when forgetting the $H_q$-module structure on each term of $K_\bullet$, the $p$-complex of $A$-modules $\mc{F}(K_\bullet)$ is null-homotopic. The null-homotopy map on $\mc{F}(K_\bullet)$, though, is not required to intertwine $H_q$-actions.

\begin{defn}\label{def-relative-homotopy-category}
Given a $p$-DG algebra $(A,\dif_A)$, the \emph{relative homotopy category} is the Verdier quotient 
$$\mc{C}^{\dif_q}(A,d_0):=\dfrac{\mc{C}(A\# H_q,d_0)}{\mathrm{Ker}(\mc{F}_d)}.$$
Likewise, the \emph{relative $p$-homotopy category} is the Verdier quotient 
$$\mc{C}^{\dif_q}(A,\dif_0):=\dfrac{\mc{C}(A\# H_q,\dif_0)}{\mathrm{Ker}(\mc{F}_\dif)}.$$
\end{defn}
The subscripts in the definitions are to remind the reader of the $H_q$-module structures on the objects.

The categories $\mc{C}^{\dif_q}(A,d_0)$ and  $\mc{C}^{\dif_q}(A,\dif_0)$ are triangulated. By construction, there are factorizations of the forgetful functors
\[
\begin{gathered}
\xymatrix{ \mc{C}(A\# H_q,d_0) \ar[rr]^{\mc{F}_d} \ar[dr] && \mc{C}(A,d_0)\\
& \mc{C}^{\dif_q}(A,d_0)\ar[ur]&
} 
\end{gathered}
\ ,
\quad
\begin{gathered}
\xymatrix{ \mc{C}(A\# H_q,\dif_0) \ar[rr]^{\mc{F}_\dif} \ar[dr] && \mc{C}(A,\dif_0)\\
& \mc{C}^{\dif_q}(A,\dif_0)\ar[ur]&
}
\end{gathered} \ .
\]

\begin{prop} \cite[Proposition 2.13]{QiSussanLink}  \label{relextot}
The $p$-extension functor $\mc{P}: \mc{C}(A\# H_q, d_0)\lra \mc{C}(A\# H_q, \dif_0)$ descends to an exact functor, still denoted $\mc{P}$, between the relative homotopy categories:
$$\mc{P}: \mc{C}^{\dif_q} (A, d_0)\lra \mc{C}^{\dif_q}(A, \dif_0) \ .$$
\end{prop}

\subsection{\texorpdfstring{$p$}{p}-DG bimodules over the polynomial algebra}\label{subset-p-DG-pol}
The graded polynomial algebra $R_n=\Bbbk[x_1,\ldots,x_n]$ has a natural $p$-DG algebra structure, where the generator $\partial_q \in H_q$ acts as a derivation determined by $\partial_q(x_i)=x_i^2$ for $i=1,\ldots,n$. Here the degree of each $x_i$ and $\dif_q$ are both two, and will be referred to as the \emph{$q$-degree}. When $n$ is clear from the context, we will abbreviate $R_n$ by just $R$.

The differential is invariant under permutation action of the symmetric group $S_n$ on the indices of the variables. Therefore let the subalgebra of polynomials symmetric in variables $x_i$ and $x_{i+1}$ with its inherited $H_q$-module structure be denoted by
\[
R^i_n=\Bbbk[x_1,\ldots,x_{i-1},x_i+x_{i+1},x_i x_{i+1},x_{i+2},\ldots,x_n].
\]
More generally, given any subgroup $G\subset S_n$, the invariant subalgebra $R_n^G$ inherits an $H_q$-algebra structure from $R_n$ (and is thus a $p$-DG algebra). In particular, we will also use the $H_q$-subalgebra 
$
R^{i,i+1}_n:= R^{S_3}_n
$,
where $S_3$ is the subgroup generated by permuting the indices $i$, $i+1$ and $i+2$.

The $(R,R)$-bimodule 
$B_i=R \otimes_{R^i} R$ has the structure of an $H_q$-module (and is thus a $p$-DG bimodule) where the differential acts via the Leibniz rule: for any $h\otimes g\in R\otimes_{R^i} R$,
$$
\partial_q(h \otimes g)=\partial_q(h) \otimes g+ h \otimes \partial_q(g).
$$
With resepct to $\otimes_R$, the tensor category of $(R,R)$-bimodules generated by the $B_i$ has an $H_q$-module structure, where the $\dif_q$ action is given by the Leibniz rule.  We denote this category by
$(R,R) \# H_q \dmod$.

Let $f=\sum_{i=1}^{n} a_ix_i \in \F_p[x_1,\dots, x_n]\subset R $ be a linear function.  We twist the $H_q$-action on the bimodule $B_i$ to obtain a bimodule $B_i^f$ defined as follows.
As an $(R,R)$-bimodule, it is the same as $B_i$ but the action of $H_q$ is twisted by defining 
\begin{subequations}
\begin{equation}\label{eqn-twistonBi-left}
  \partial_q(1 \otimes 1)=(1 \otimes 1)f.
\end{equation}
Similarly we define ${}^f B_i$ where now 
\begin{equation}\label{eqn-twistonBi-right}
    \partial_q(1\otimes1)=f(1 \otimes 1).
\end{equation}
\end{subequations}

For $R_n$ as a bimodule over itself, it is clear that $^fR_n \cong R_n^f$ as $p$-DG bimodules.
It follows that there are $p^n$ ways to put an $H_q$-module structure on a rank-one free module over $R_n$.
Each such $H_q$-module is quasi-isomorphic to a finite-dimensional $p$-complex. Choose numbers $b_i\in \{2 ,  \dots, p, p+1 \}$ such that $b_i\equiv a_i~(\mathrm{mod}~p)$, $i=1,\dots, n$, and define the $H_q$-ideal of $R$ 
\begin{equation}
I=(x_1^{p+1-b_1},\cdots, x_{n}^{p+1-b_n}).
\end{equation}
Then the natural quotient map
\begin{equation} \label{eqn-Rf-slashi-homology}
\pi:R^f \twoheadrightarrow R^f/(I\cdot R^f)
\end{equation}
is readily seen to be a quasi-isomorphism. The right hand side of \eqref{eqn-Rf-slashi-homology} computes the \emph{slash homology} (see \cite[Section 2.1]{QiSussanLink} for more details), denoted $\mH_\bullet^{/}$, of $R^f$.

\begin{lem} \cite[Lemma 3.1]{QiSussanLink}\label{lem-pol-mod}
For each $f=\sum_i a_ix_i$, the rank-one $p$-DG module $R^f$ has finite-dimensional slash homology:
\[
\mH^/_{\bullet}(R^f)\cong \bigotimes_{i=1}^n V_{p-a_i} \{p-a_i\} .
\]
In particular, if any $a_i$ of $f=\sum_i a_ix_i$ is equal to one, then $\mH^/_{\bullet}(R^f)=0$.
\end{lem}

\begin{cor} \cite[Corollary 3.2]{QiSussanLink} \label{cor-finite-slash-homology}
Let $M$ be a $p$-DG module over $R$ which is equipped with a finite filtration, whose subquotients are isomorphic to $R^f$ for various $f$. Then $M$ has finite-dimensional slash homology.
\end{cor}

\subsection{Relative \texorpdfstring{$p$}{p}-Hochschild homology}
In \cite[Section 2.3]{QiSussanLink}, we introduced an absolute version of the $p$-Hochschild (co)homology functor. In what follows, we will instead need a relative version of $p$-Hochshild homology for a $p$-DG algebra, which we recall now. An important reason for introducing the relative homotopy category is that the relative $p$-Hochschild homology functor descends to this category.

Let $(A,\dif_A)$ be a $p$-DG algebra. Equip $A$ with the zero differential $d_0$ and zero $p$-differential $\dif_0$, and denote the resulting trivial ($p$)-DG algebras by $(A_0,d_0)$ and $(A_0,\dif_0)$ respectively. Likewise, for a ($p$-)DG bimodule $M$ over $A$, we temporarily denote by $M_0$ the $A$-bimodule equipped with zero ($p$-) differentials.

The usual Hochschild homology of $M_0$ over $(A_0,d_0)$ in this case carries a natural $H_q$-action, since the $H_q$-action commutes with all differentials in the usual simplicial bar complex for $A_0$. 

\begin{defn}
The \emph{relative Hochschild homology} of a $p$-DG bimodule $(M,\dif_M)$ over $(A,\dif_A)$ is the usual Hochschild homology of $M_0$ over $(A_0,d_0)$ equipped with the induced $H_q$-action from $\dif_M$ and $\dif_A$, and denoted
\[
\mHH^{\dif_q}_\bullet(M):=\mHH_\bullet(A_0,M_0) \ .
\]
\end{defn}

Replacing the usual simplicial bar complex by Mayer's $p$-simplicial bar complex (see \cite[Definition 2.10]{QiSussanLink}, essentially, one just needs to remove the signs in the usual simplicial bar complex to obtain a $p$-complex resolution), we make the following definition (see \cite[Section 2.3]{QiSussanLink} for details).

\begin{defn}
The \emph{relative $p$-Hochschild homology} of $M$ is the $p$-complex 
\[
\pHH^{\dif_q}_\bullet (M):=\mH^/_{\bullet}(A_0 \otimes_{A_0\otimes A_0^{\mathrm{op}}}^{\mathbf{L}} M_0)=\mH^/_{\bullet}(\mathbf{p}(A_0) \otimes_{A_0\otimes A_0^{\mathrm{op}}} M_0)
\ .
\]
Here, the usual simplicial bar resolution of $M_0$ over $A_0$ is replaced by \emph{Mayer's $p$-simplicial bar complex} $\mathbf{p}(A_0)$.
\end{defn}

Similar to the usual Hochschild homology, the relative $p$-Hochschild homology is also covariant functor: if $f: M \lra N$ is a morphism of $p$-DG bimodules over $A$, it induces
\[
\pHH_\bullet^{\dif_q}( f ):=\mH^/_{\bullet}(\mathrm{Id}_{A_0}\otimes f): \mH^/_{\bullet}(A_0 \otimes_{A_0\otimes A_0^{\mathrm{op}}}^{\mathbf{L}} M_0) 
\lra \mH^/_{\bullet}(A_0 \otimes_{A_0\otimes A_0^{\mathrm{op}}}^{\mathbf{L}} N_0)
\ .
\]

\begin{prop} \cite[Proposition 2.20]{QiSussanLink}
The relative $p$-Hochschild homology descends to a functor defined on the relative homotopy category $\mc{C}^{\dif_q}(A,\dif_0)$ of $p$-DG bimodules over $A$.
\end{prop}

We also have the trace-like property for relative $p$-Hochschild homology.

\begin{prop} \cite[Proposition 2.21]{QiSussanLink}  \label{HHrelativecyclprop}
Given two $p$-DG bimodules $M$ and $N$ over $A$, there is an isomorphism of $p$-complexes of $H_q$-modules
\[
\pHH^{\dif_q}_\bullet(M\otimes^{\mathbf{L}}_A N)\cong \pHH^{\dif_q}_\bullet(N\otimes^{\mathbf{L}}_A M).
\]
\end{prop}

Our next goal is to recall a technical tool that allows us to use a simpler bimodule resolution to compute the relative Hochschild homology than the usual simplicial bar resolution. 

\begin{thm} \cite[Theorem 2.22]{QiSussanLink} \label{thm-resolution-independence}
Let $M$ be a $p$-DG bimodule over $A$. Suppose $f:Q_\bullet \lra M$ is a $p$-complex resolution of $M$ over $(A_0,\dif_0)$ which is $H_q$-equivariant, and each term of $Q_\bullet$ is projective as an $A_0\otimes A_0^{\mathrm{op}}$-module. Then $f$ induces an isomorphism of $H_q$-modules
\[
\mH^{/}_\bullet(A_0\otimes_{A_0\otimes A_0^{\mathrm{op}}}Q_\bullet)\cong \pHH^{\dif_q}_\bullet(M).
\]
\end{thm}

\subsection{Elementary braiding complexes}
In \cite{QiSussanLink}, we showed that there are $(R,R) \# H_q$-module homomorphisms
\begin{enumerate}
\item[(i)] $rb_i \colon R \longrightarrow q^{-2} B_i^{-(x_i+x_{i+1})}$, where $1 \mapsto (x_{i+1} \otimes 1 
- 1 \otimes x_{i}) $;
\item[(ii)] $br_i \colon B_i \longrightarrow R$, where $1 \otimes 1 \mapsto 1$.
\end{enumerate}

Thus we have complexes of $(R,R) \# H_q$-modules 
\begin{equation}\label{eqn-elementary-braids}
T_i :=
\left(t B_i  \xrightarrow{br_i} R\right)
,
\quad \quad \quad
T_i' :=  \left(R \xrightarrow{rb_i} q^{-2} t^{-1} B_i^{-(x_i+x_{i+1})}\right)
.
\end{equation}
In the coming sections we will, for presentation reasons, often omit the various shifts built into the definitions of $T_i$ and $T_i'$.

We associate respectively to the left and right crossings  $\sigma_i$ and $\sigma_i^{\prime}$ between the $i$th and $(i+1)$st strands in \eqref{2crossings} the chain complexes of $(R,R)\#H_q$-bimodules $T_i$ and $T_i'$: 
\begin{equation} \label{2crossings}
\sigma_i:=
\begin{DGCpicture}
\DGCPLstrand(-1,0)(-1,1)
\DGCPLstrand(0,0)(1,1)
\DGCPLstrand(1,0)(.75,.25)
\DGCPLstrand(0,1)(.25,.75)
\DGCPLstrand(2,0)(2,1)
\DGCcoupon*(-1,0)(0,1){$\cdots$}
\DGCcoupon*(1,0)(2,1){$\cdots$}
\end{DGCpicture}
\hspace{1in}
\sigma_i^\prime:=
\begin{DGCpicture}
\DGCPLstrand(-1,0)(-1,1)
\DGCPLstrand(0,0)(.25,.25)
\DGCPLstrand(.75,.75)(1,1)
\DGCPLstrand(1,0)(0,1)
\DGCPLstrand(2,0)(2,1)
\DGCcoupon*(-1,0)(0,1){$\cdots$}
\DGCcoupon*(1,0)(2,1){$\cdots$}
\end{DGCpicture}
\end{equation}
More generally, if $\beta\in \mathrm{Br}_n$ is a braid group element written as a product in the elementary generators $\sigma_{i_i}^{\epsilon_1}\cdots \sigma_{i_k}^{\epsilon_k}$, where $\epsilon_i\in \{\emptyset, \prime\}$, we assign the chain complex of $(R,R)\# H_q$-bimodules
\begin{equation}
    T_\beta:=T_{i_1}^{\epsilon_1}\otimes_R\cdots \otimes_R T_{i_k}^{\epsilon_k}.
\end{equation}

\begin{thm}\label{thm-braid-invariant}
The complexes of $T_i$, $T_i^\prime$ are mutually inverse complexes in the relative homotopy category $\mc{C}^{\dif_q}(R,R,d_0)$. They satisfy the braid relations
\begin{itemize}
    \item $T_iT_j\cong T_jT_i$ if $|i-j|>1$,
    \item $T_iT_{i+1} T_i \cong T_{i+1}T_i T_{i+1}$ for all $i=1,\dots, n-1$.
\end{itemize}
Consequently, given any braid group element $\beta\in \mathrm{Br}_n$, the chain complex of $T_\beta$ associated to it is
a well defined element of the relative homotopy category $\mc{C}^{\dif_q}(R,R,d_0)$.
\end{thm}
\begin{proof}
This is proven in \cite[Section 3]{QiSussanLink}.
\end{proof}

\section{Specialized HOMFLYPT theories} \label{sechomflypt}
\subsection{HOMFLYPT homologies}
In this section we categorify the HOMFLYPT polynomial of any link using analogous arguments from  \cite{Cautisremarks}, \cite{RW} and \cite{Roulink} adapted to the $p$-DG setting. 

For the first construction, we will allow complexes of Soergel bimodules to sit in half-integer degrees in the Hochschild ($a$) and topological ($t$) degrees when considering the usual complexes of vector spaces.
We modify the elementary braiding complexes of equation \eqref{eqn-elementary-braids} to be
\begin{equation}\label{eqn-elementary-braids-half-grading}
T_i :=
 (at)^{-\frac{1}{2}}q^{-2} \left(t B_i  \xrightarrow{br_i} R\right)
,
\quad \quad \quad
T_i' := (at)^{\frac{1}{2}} q^{2} \left(R \xrightarrow{rb_i} q^{-2} t^{-1} B_i^{-(x_i+x_{i+1})}\right)
.
\end{equation}
Here we have extended the degree shift convention for $q$-degrees (see the beginning of Section \ref{secbackground}) to $a$ and $t$-degrees.

Let $\beta\in \mathrm{Br}_n$ be a braid group element in $n$ strands. By Theorem \ref{thm-braid-invariant}, there is a chain complex of $(R_n,R_n)\# H_q$-bimodules $T_\beta$, well defined up to homotopy, associated with $\beta$. Then set
\begin{equation}\label{eqn-chain-complex-for-braid}
    T_\beta = \left(\dots\stackrel{d_0}{\lra} T_\beta^{i+1} \stackrel{d_0}{\lra} T_\beta^{i} \stackrel{d_0}{\lra} T_\beta^{i-1}\stackrel{d_0}{\lra}\dots\right).
\end{equation}

\begin{defn}\label{def-HHH}
The \emph{untwisted $H_q$-HOMFLYPT homology} of $\beta$ is the object 
\[
\mtHHH^{\dif_q}(\beta):=a^{-\frac{n}{2}}t^{\frac{n}{2}}\mH_\bullet \left(\dots \lra \mHH_\bullet^{\dif_q}(T_\beta^{i+1}) \xrightarrow{d_t} \mHH_\bullet^{\dif_q}(T_\beta^{i}) \xrightarrow{d_t}  \mHH_\bullet^{\dif_q}(T_\beta^{i-1})\lra \dots\right)
\]
in the category of triply-graded $H_q$-modules, where $d_t:=\mHH_\bullet^{\dif_q}(d_0)$ is the induced map of $d_0$ on Hochschild homology.  
\end{defn}

By construction, the space $\mtHHH^{\dif_q}(\beta)$ is triply graded by topological ($t$) degree, Hochschild ($a$) degree as well as quantum ($q$) degree. When necessary to emphasize each graded piece of the space, we will write
$\mtHHH^{\dif_q}_{i,j,k}(\beta)$ to denote the homogeneous component concentrated in $t$-degree $i$, $a$-degree $j$ and $q$-degree $k$.

The following theorem is a particular case of the main result of \cite{KRWitt}, where we have only kept track of the degree two $p$-nilpotent differential (which is denoted $L_1$ in \cite{KRWitt}) in finite characteristic $p$. The detailed verification given below, however, uses the main ideas of \cite{Roulink} and differs from that of \cite{KRWitt}. This proof serves as the model for the other link homology theories in this paper. 

\begin{thm}\label{thm-untwisted-HOMFLY}
The untwisted $H_q$-HOMFLYPT homology of $\beta$ depends only on the braid closure of $\beta$ as a framed link in $\R^3$.
\end{thm}

As a convention for the framing number of braid closure, if a strand for a component of link is altered as in the left of \eqref{framing}, then we say that the framing of the component is increased by $1$ (with respect to the blackboard framing).
If a strand for a component of link is altered as in the right of \eqref{framing}, then we say that the framing of the component is decreased by $1$.

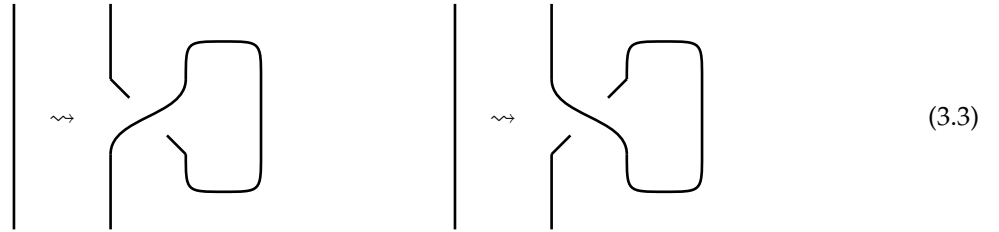
\begin{equation} \label{framing}
\begin{DGCpicture}
\DGCPLstrand(-2,-1)(-2,2)
\end{DGCpicture}
\quad
\rightsquigarrow
\quad
\begin{DGCpicture}
\DGCstrand(0,0)(1,1)
\DGCPLstrand(1,0)(.75,.25)
\DGCPLstrand(0,1)(.25,.75)
\DGCstrand(1,1)(1.5,1.5)(2,1)(2,0)(1.5,-.5)(1,0)
\DGCPLstrand(0,1)(0,2)
\DGCPLstrand(0,0)(0,-1)
\end{DGCpicture}
\hspace{1in}
\begin{DGCpicture}
\DGCPLstrand(-2,-1)(-2,2)
\end{DGCpicture}
\quad
\rightsquigarrow
\quad
\begin{DGCpicture}
\DGCPLstrand(0,0)(.25,.25)
\DGCPLstrand(.75,.75)(1,1)
\DGCstrand(1,1)(1.5,1.5)(2,1)(2,0)(1.5,-.5)(1,0)
\DGCstrand(1,0)(0,1)
\DGCPLstrand(0,1)(0,2)
\DGCPLstrand(0,0)(0,-1)
\end{DGCpicture}
\end{equation}
Denote by $\mathtt{f}_i(L)$ the framing number of the $i$th component of a link $L$. Then, under the Reidemeister moves of \eqref{framing}, $\mathtt{f}_i(L)$ is increased or decreased by one when changing from the corresponding left local picture to the right local picture.

We next seek to define a triply graded analogue with $a$, $t$ and $q$-degrees in the homotopy category of $p$-complexes. Let us first discuss what degrees of freedom we have in the constructions.

Firstly, we may adapt \eqref{eqn-elementary-braids-half-grading} into
\begin{subequations}\label{eqn-elementary-braids-general}
\begin{align}
    pT_i& :=a^u t^v  q^w [n]^a_\dif[m]^t_\dif \left(
    B_i[1]^t_\dif\xrightarrow{br_i} R
    \right),   \\
    pT_i^\prime  & :=a^{-u}t^{-v}q^{-w}  [-n]^a_\dif   [-m]^t_\dif\left(
    R\xrightarrow{rb_i} q^{-2}B_i^{-x_i-x_{i+1}}[-1]^t_\dif
    \right).
\end{align}
\end{subequations}
Here, the superscripts in homological shifts indicate in which of the three gradings they are occurring. We let $u, v, w, m, n\in \Z$ denote possible shifts to be determined, which will be made into the simplest possible form at the end of the next subsection.

\begin{defn}\label{def-pTbeta}
Let $\beta\in \mathrm{Br}_n$ be a braid group element written as a product in the elementary generators $\sigma_{i_i}^{\epsilon_1}\cdots \sigma_{i_k}^{\epsilon_k}$, where $\epsilon_i\in \{\emptyset, \prime\}$. We assign to $\beta$ the $p$-chain complex of $(R_n,R_n)\# H_q$-bimodules
\begin{equation}
    pT_\beta:=pT_{i_1}^{\epsilon_1}\otimes_R\cdots \otimes_R pT_{i_k}^{\epsilon_k}.
\end{equation}
\end{defn}

We will denote the boundary maps in the $p$-complex $\pT_\beta$ by $\dif_0$, in contrast to the usual topological differential $d_0$.

\begin{defn}\label{def-pHHH}
The \emph{ untwisted $H_q$-HOMFLYPT $p$-homology } of $\beta$ is the object 
\[
 \ptHHH^{\dif_q}(\beta):= q^{f( n ) } \mH^/_{\bullet} \left(\dots \lra \pHH_\bullet^{\dif_q}(\pT_\beta^{i+1}) \xrightarrow{\dif_t}  \pHH_\bullet^{\dif_q}(\pT_\beta^{i}) \xrightarrow{\dif_t}  \pHH_\bullet^{\dif_q}(\pT_\beta^{i-1})\lra \dots\right)
\]
in the homotopy category of bigraded $H_q$-modules, where $f(n)$ is a function on $\N$ to be determined. Here $\dif_t$ stands for the induced map of the topological differentials on $p$-Hochschild homology groups $\dif_t:=\pHH_\bullet^{\dif_q}(\dif_0)$.
\end{defn}

In the definition of the $H_q$-HOMFLYPT $p$-homology, we have applied the $p$-extensions in both the topological and the Hochschild direction so that they can be collapsed into a single degree. The reason will become clearer later when categorifying certain $\mathfrak{sl}_N$ polynomials at prime roots of unity. Therefore, in contrast to $\mtHHH(\beta)$, $\ptHHH(\beta)$ is only doubly-graded, and we will adopt the notation $\ptHHH_{i,j}(\beta)$ as above to stand for its homogeneous components in topological degree $i$ and $q$ degree $j$. Further, the overall grading shift in the definition will be utilized in the invariance under the Markov II move below.

\begin{thm}\label{thm-untwisted-HOMFLY-p-ext}
The untwisted $H_q$-HOMFLYPT $p$-homology of $\beta$ depends only on the braid closure of $\beta$ as a framed link in $\R^3$.
\end{thm}

The proof of Theorems \ref{thm-untwisted-HOMFLY} and \ref{thm-untwisted-HOMFLY-p-ext} will occupy the next few subsections, after we introduce the $H_q$-equivariant ($p$-)Koszul resolutions.

\subsection{Re-examining Markov II invariance} \label{secmarkov2}
In this subsection, let us re-examine invariance under the Markov II move for $\ptHHH$  under these assumptions.

In order to satisfy the second Markov move, one needs to show that for a Soergel bimodule $M$ (or a complex of Soergel bimodules) over the polynomial $p$-DG algebra $R_n$, that the $H_q$-HOMFLYPT ($p$-)homologies of the bimodules \eqref{markov2pic} are isomorphic (up to shifts and twists).
\begin{equation}
\label{markov2pic}
\begin{DGCpicture}[scale=0.8]
\DGCstrand(0,0)(0,3)
\DGCstrand(2,0)(2,3)
\DGCcoupon(-0.5,1)(2.5,2){$M$}
\DGCcoupon*(0,0)(2,1){$\cdots$}
\DGCcoupon*(0,2)(2,3){$\cdots$}
\end{DGCpicture}
\quad \quad \quad \quad \quad \quad 
\begin{DGCpicture}[scale=0.8]
\DGCstrand(0,0)(0,-3)
\DGCstrand(2,0)(2,-1.7)(2.4,-2.5)/dr/
\DGCstrand(3,0)(3,-2)(2,-3)/dl/
\DGCPLstrand(2.6,-2.7)(3,-3)
\DGCcoupon(-0.25,-0.5)(2.25,-1.5){$M$}
\DGCcoupon*(0,0)(2,-0.5){$\cdots$}
\DGCcoupon*(0,-2)(2,-3){$\cdots$}
\end{DGCpicture}
\end{equation}

By definition, the one-variable $p$-extended Koszul complex is given by
\begin{equation}
    pC_1=q^2 \Bbbk[x]^x\otimes \Bbbk[x]^x[1]^a_\dif\xrightarrow{x\otimes 1 - 1\otimes x} \Bbbk[x]\otimes \Bbbk[x] .
\end{equation}
Set $\pC_{n+1}:=\pC_1^{\otimes n+1}$. For the ease of notation, we will write $\pC_1^\prime$ for the $p$-extended Koszul complex $\pC_1$ in the variable $x_{n+1}$. Using the isomorphism of $p$-DG bimodules
\begin{align}
\pC_{n+1} \otimes_{(R_{n+1},R_{n+1})} ((M \otimes \Bbbk[x_{n+1}]) \otimes_{R_{n+1}} \pT_n) &=
(\pC_n \otimes \pC_1^\prime) \otimes_{(R_{n+1},R_{n+1})} ((M \otimes \Bbbk[x_{n+1}]) \otimes_{R_{n+1}} \pT_n) \nonumber \\
& \cong 
\pC_n \otimes_{(R_n,R_n)} (M \otimes_{R_n} (\pC_1^\prime \otimes_{(\Bbbk[x_{n+1}],\Bbbk[x_{n+1}])} \pT_n)),
\end{align}
we are reduced to analyzing the $p$-homology of the ``cubes'' $pC_1^\prime \otimes_{\Bbbk[x_{n+1}],\Bbbk[x_{n+1}]} pT_n$:
\begin{subequations}
\begin{align}\label{eqn-cube-1}
\begin{gathered}
\xymatrix{
a^ut^vq^w({}^{x_{n+1}}B_i^{x_{n+1}})[k+1]^a_\dif[m+1]^t_\dif \ar[rr]\ar[dd] && a^ut^vq^w R^{2x_{n+1}}[k+1]^a_\dif[m]^t_\dif \ar[dd]\\
&& \\
a^ut^vq^w B_n [k]^a_\dif[m+1]^t_\dif  \ar[rr] && a^ut^vq^w R [k]^a_\dif[m]^t_\dif
}
\end{gathered}  \ ,
\end{align}
and $pC_1^\prime \otimes_{\Bbbk[x_{n+1}],\Bbbk[x_{n+1}]} pT_n^\prime$:
\begin{align}\label{eqn-cube-2}
\begin{gathered}
\xymatrix{
a^{-u}t^{-v}q^{2-w}R^{2x_{n+1}}[1-k]^a_\dif[-m]^t_\dif \ar[rr]\ar[dd] && a^{-u}t^{-v}q^{2-w}(^{x_{n+1}}B_n^{-x_n})[1-k]^a_\dif[-1-m]^t_\dif \ar[dd]\\
&& \\
a^{-u}t^{-v}q^{-w}[-k]^a_\dif[-m]^t_\dif \ar[rr] && a^{-u}t^{-v}q^{-w-2}B_n^{-x_i-x_{i-1}}[-k]^a_\dif[-1-m]^t_\dif
}
\end{gathered}  \ .
\end{align}
\end{subequations}

Let us begin by studying the first $p$-complex cube \eqref{eqn-cube-1}. Ignoring for the moment the overall grading shift $a^{u}t^vq^w[k]^a_\dif[m]^t_\dif$ for simplicity, we have a filtration of the cube given by a short exact sequence of ($p$-complexes) of bimodules
\begin{equation} \label{sesbicomplexespHHH}
\begin{gathered}
\xymatrix@R=1em@C=1.4em{
& & & & & 0 \ar[dd] \\
  q^4 R_{n+1}^{x_n+3x_{n+1}}  [1]_\dif^a [1]_\dif^t\ar@/_4pc/[dddd]_{\phi} \ar[rrr]^-{2(x_{n+1}-x_n)} \ar[dd] & & &  q^2 R_{n+1}^{2x_{n+1}} [1]_\dif^a \ar[dd] \ar@/^2pc/[dddd]^{\Id} &  \\
& & & & := & pY_1 \ar[dddd]\\
0 \ar[rrr] & & & 0 \\ 
& & &  \\
q^2 ({}^{x_{n+1}}B_n^{x_{n+1}}) [1]_\dif^a[1]_\dif^t \ar[rrr]^{br} \ar[dd]^{x_{n+1} \otimes 1 - 1 \otimes x_{n+1}} \ar@/_4pc/[dddd]_{\tilde{br}} & & & q^{2} R_{n+1}^{2x_{n+1}} [1]_\dif^a  \ar[dd]^{0} &  & \\
& & &  & = &
\pC_1^\prime \otimes_{(\Bbbk[x_{n+1}],\Bbbk[x_{n+1}])} \pT_n \ar[dddd] \\
  B_n [1]_\dif^t\ar[rrr]^{br} \ar@/_4pc/[dddd]_{2 \Id} & & &  R_{n+1} \ar@/^2pc/[dddd]^{2\Id} \\
&  &  & \\
q^{2} \tilde{R}_{n+1}^{x_n+x_{n+1}} [1]_\dif^t \ar[rrr] \ar[dd]^{(x_{n+1}-x_n) \otimes 1 - 1 \otimes (x_{n+1}-x_n)} & & & 0 \ar[dd] &  & \\
& & & & := & pY_2 \ar[dd] \\
B_n [1]_\dif^t \ar[rrr]^{br} & & &   R_{n+1}\\
& & & & & 0  
}
\end{gathered} \ .
\end{equation}
Here $\phi$ is the map that sends $1$ to $(x_{n+1}-x_n) \otimes 1 + 1 \otimes (x_{n+1}-x_n)$.

It is not hard to show that $\pC_n \otimes_{(R_n,R_n)} (M \otimes_{R_n} pY_2)$ is annihilated by taking first the vertical $p$-Hochschild homology and then the horizontal topological homology. Further, the $p$-complex $pY_1$ is quasi-isomorphic to $q^2 R_n^{2x_n} [1]_\dif^a$. Recovering the grading shifts, we obtain the isomorphism
 \begin{align}\label{eqn-general-twisting-1}
 \ptHHH^{\dif_q}((M\otimes \Bbbk[x_{n+1}]) \otimes_{R_{n+1}} \pT_n) & \cong  \mH^/_\bullet (\pHH^{\dif_q}((M\otimes_{R_n} pY_1 ))\nonumber \\
  & \cong \ptHHH^{\dif_q}(a^ut^vq^{w+2}M[k+1]^a_\dif[m]^t_\dif)^{2x_n}.
 \end{align}

For the second cube \eqref{eqn-cube-2}, again there is a short exact sequence of bicomplexes of $(R_{n+1},R_{n+1})$-bimodules. Ignoring the overall grading shift $a^{-u}t^{-v}q^{-w}[-k]^a_\dif[-m]^t_\dif$, it is given by
\begin{equation} \label{sesbicomplexes2pHHH}
\begin{gathered}
\xymatrix@R=1em@C=2em{
& & & & & 0 \ar[dd] \\
q^2 R_{n+1}^{2x_{n+1}}[1]^a_\dif \ar@/_2pc/[dddd]_{\Id} \ar[rrr]^{\Id} \ar[dd] & & & q^2 R_{n+1}^{2x_{n+1}}[1]^a_\dif[-1]^t_\dif \ar[dd] \ar@/^5pc/[dddd]^{rb} &  & \\
& & & & := & pZ_1\ar[dddd] \\
0 \ar[rrr] & & & 0 \\ 
& & & & &  \\
q^2 R_{n+1}^{2x_{n+1}} [1]^a_\dif  \ar[rrr]^{rb} \ar[dd]^{0}  & & &   ( {}^{x_{n+1}}B_{n}^{-x_n})[1]^a_\dif[-1]^t_\dif \ar[dd]_{x_{n+1} \otimes 1 - 1 \otimes x_{n+1}} \ar@/^5pc/[dddd]^{\tilde{br}} & & \\
& & & & = & 
pC_1^\prime \otimes_{(\Bbbk[x_{n+1}],\Bbbk[x_{n+1}])} pT_n' \ar[dddd] \\
   R_{n+1}  \ar[rrr]^{rb} \ar@/_2pc/[dddd]_{ \Id} & & &  q^{-2} B_{n}^{-(x_n+x_{n+1})}[-1]^t_\dif \ar@/^5pc/[dddd]^{\Id} \\
& & & & & \\
0 \ar[rrr] \ar[dd]^{} & & &  \tilde{R}_{n+1}[1]^a_\dif[-1]^t_\dif \ar[dd]_{x_{n+1} \otimes 1 - 1 \otimes x_{n+1}} &  &   \\
 & & & & := & pZ_2\ar[dd]\\
  R_{n+1}\ar[rrr]^{rb} & & &  q^{-2} B_{n}^{-(x_n+x_{n+1})} [-1]^t_\dif \\
& & & & & 0  
} \ .
\end{gathered}
\end{equation}
Getting rid of contractible summands, we see that $pZ_2$ is homotopy equivalent to
\begin{equation}\label{eqn-twisting-factor-pHHH-2}
    R_{n+1} \xrightarrow{x_{n+1}-x_n}  q^{-2} R_{n+1}^{-(x_n+x_{n+1})}[-1]^t_\dif,
\end{equation}
which is, in turn, quasi-isomorphic to $q^{-2} R_n^{-2x_n}[-1]^t_\dif$. Taking back into account the overall grading shift, we have
 \begin{align}\label{eqn-general-twisting-2}
 \ptHHH^{\dif_q}((M\otimes \Bbbk[x_{n+1}]) \otimes_{R_{n+1}} \pT_n^\prime) & \cong  \mH^/_\bullet (\pHH^{\dif_q}((M\otimes_{R_n} pZ_2 )) \nonumber \\
 & \cong \ptHHH^{\dif_q}(a^{-u}t^{-v}q^{-w-2}M[-k]^a_\dif[-m-1]^t_\dif)^{-2x_n}.
 \end{align}
 
 Now, let us observe that taking closure of the following diagram of $p$-DG bimodules
 \begin{equation}
    \begin{DGCpicture}[scale=0.8]
    \DGCstrand(0,0)(0,-4)
    \DGCstrand(2,0)(2,-1.75)(2.75,-2.75)/dr/
    \DGCstrand/dr/(3,-3)(4,-4)/dr/
    \DGCstrand(3,0)(3,-2.5)(2,-4)/dl/
    \DGCstrand(4,0)(4,-2)(3.65,-3.35)/dl/
    \DGCPLstrand(3.45,-3.65)(3,-4)
    \DGCcoupon(-0.2,-0.5)(2.2,-1.5){$M$}
    \DGCcoupon*(0,0)(2,-0.5){$\cdots$}
    \DGCcoupon*(0,-1.5)(2,-4){$\cdots$}
    \end{DGCpicture}
\end{equation}
introduces a cancelling pair of Markov II moves. By equations \eqref{eqn-general-twisting-1} and \eqref{eqn-general-twisting-2}, we obtain that
\begin{equation}
    \ptHHH^{\dif_q}\left((M\otimes \Bbbk[x_{n+1},x_{n+2}])\otimes_{R_{n+2}}pT_{n+1}\otimes_{R_{n+2}}pT_{n+2}^\prime \right) \cong 
    q^{f(n+2)} \mH^{/}_\bullet (\pHH_\bullet (M[1]^{a}_\dif[-1]^t_\dif))
\end{equation}
For the last term to be isomorphic to
\begin{equation}
    \ptHHH^{\dif_q}(M)\cong q^{f(n)}\mH_\bullet^{/}(\pHH_\bullet (M)),
\end{equation}
we need to require the functor isomorphism
\begin{equation}
    [1]^a_\dif[-1]^t_\dif=q^{f(n)-f(n+2)} .
\end{equation}
We are therefore forced to collapse the $a$ grading onto the $t$ grading in the way $a=q^rt$, where $r=f(n)-f(n+2)\in \Z$. For simplicity, let us assume that 
\begin{equation} \label{defofK}
K=f(n)-f(n+1)
\end{equation}
is a constant independent of $n$. Then $r=2K\in 2\Z$, 
and we have $a=q^{2K}t$ so that $[1]^a_\dif=q^{2K}[1]^t_\dif$.

Revisiting \eqref{eqn-elementary-braids-general}, we now set  

\begin{subequations}\label{eqn-elementary-braids-general2}
\begin{align}
    pT_i& := q^{-K-2} [-1]^t_\dif \left(
    B_i[1]^t_\dif\xrightarrow{br_i} R
    \right),   \\
    pT_i^\prime  & :=q^{K+2} [1]^t_\dif\left(
    R\xrightarrow{rb_i} q^{-2}B_i^{-x_i-x_{i+1}}[-1]^t_\dif
    \right),
\end{align}
\end{subequations}
and
\[
 \ptHHH^{\dif_q}(\beta,K+1):= q^{-Kn } \mH^/_{\bullet} \left(\cdots \lra \pHH_\bullet^{\dif_q}(\pT_\beta^{i+1}) \xrightarrow{\dif_t}  \pHH_\bullet^{\dif_q}(\pT_\beta^{i}) \xrightarrow{\dif_t}  \pHH_\bullet^{\dif_q}(\pT_\beta^{i-1})\lra \cdots\right).
\]

\begin{thm}
Let $\beta_1$ and $\beta_2$ be two braids whose closures represent the same link $L$ of $r$ components up to framing.  Suppose the framing numbers of the closures $\widehat{\beta}_1$ of $\beta_1$ and $\widehat{\beta}_2$ of $\beta_2$ differ by $\mathtt{f}_i(\widehat{\beta}_1)-\mathtt{f}_i(\widehat{\beta}_2)=a_i$, $i=1,\dots, r$. Then 
\begin{equation*}
\mtHHH^{\dif_q}(\beta_1) \cong \mtHHH^{\dif_q}(\beta_2)^{2\sum_{i=1}^r a_i x_i}
\end{equation*}
and
\begin{equation*}
p \mtHHH^{\dif_q}(\beta_1,K+1) \cong p \mtHHH^{\dif_q}(\beta_2,K+1)^{2 \sum_{i=1}^r a_i x_i}
\end{equation*}
where the generator of the polynomial action for the $i$th component is denoted $x_i$ and $\mtHHH^{\dif_q}(\beta_2)^{2\sum_i a_i x_i}$ means that we twist the $H_q$-module structure on the $i$th component by $2a_ix_i$.
\end{thm}

\begin{proof}
The topological invariance follows from Theorem \ref{thm-braid-invariant} and the proof of invariance under the Markov moves.
\end{proof}



\subsection{Unlinks and twistings}\label{subsecHOMFLYunlink}
In this section, we compute $\mHHH^{\dif_q}$ and $\pHHH^{\dif_q}$ for the identity element of the braid group $\mathrm{Br}_n$, and define an unframed link invariant in $\R^3$.

For the unknot, the Koszul resolution $C_1$ of $\Bbbk[x]$ as bimodules given by
\begin{equation*}
\xymatrix{
q^2a\Bbbk[x]^{x} \otimes \Bbbk[x]^{x}   \ar[rr]^-{x\otimes 1-1\otimes x} &&
\Bbbk^{}[x] \otimes \Bbbk^{}[x]
}
\ .
\end{equation*}
Tensoring this complex with $\Bbbk[x]$ as a bimodule yields 
\begin{equation*}
\xymatrix{
q^2a\Bbbk[x]^{2x}  \ar[r]^{\hspace{.2in} 0} & \Bbbk[x]
}    
\ .
\end{equation*}
Thus the homology of the unknot (up to shift) is identified with the bigraded $H_q$-module
\begin{equation*}
\Bbbk[x]\oplus q^2 a\Bbbk[x]^{2x}  
\ .
\end{equation*}

More generally, via the Koszul complex $C_n=C_1^{\otimes n}$, we have that the homology of the $n$-component unlink $L_0$   is equal to
\begin{equation}
   \mtHHH^{\dif_q}(L_0)\cong  a^{-\frac{n}{2}}t^{\frac{n}{2}} \mHH_\bullet(R_n) \cong a^{-\frac{n}{2}}t^{\frac{n}{2}} \bigotimes_{i=1}^n \left( \Bbbk[x_i] \oplus q^2 a \Bbbk[x_i]^{2x_i} \right).
\end{equation}
Alternatively, up to the grading shift $ a^{-\frac{n}{2}}t^{\frac{n}{2}} $,  we may identify $\mtHHH^{\dif_q}(L_0)$ with the exterior algebra over $R_n$ generated by the differential forms $dx_i$ of bidegree $aq^2$, $i=1,\dots, n$, subject to the condition that each $dx_i$ accounts for a twisting of $H_q$-module structure by $2x_i$.

It follows that, as for the ordinary HOMFLYPT homology, given a framed link $L$ of $\ell $ components arising as a braid closure $\widehat{\beta}$, its untwisted HOMFLYPT $H_q$-homology $\mtHHH^{\dif_q} (\beta) $ is a module over 
$$\mtHHH_{0,0,\bullet}^{\dif_q}(L_0)\cong R_\ell,$$ 
and thus one may consider a twisting of the $H_q$-module structure on $\mtHHH^{\dif_q}(\beta)$ by via the functor
$ R_\ell ^f\otimes_{R_\ell}(\mbox{-}) $, where $f$ is a linear polynomial in $x_1, \dots , x_\ell$, (see Section \ref{subset-p-DG-pol}).

\begin{defn}\label{def-twisted-HHH}
Let $L$ be a framed link arising from the closure of an $n$-strand braid $\beta$. Label the components of $L$ by $1$ through $\ell$, and set the \emph{(linear) framing factor} of $\beta$ to be the linear polynomial
\[
\mathtt{f}_\beta =- \sum_{i=1}^\ell 2 \mathtt{f}_i x_i.
\]
\begin{enumerate}
\item[(1)]The \emph{$H_q$-HOMFLYPT homology} of $\beta$ is the triply-graded $H_q$-module
\[
\mHHH^{\dif_q}(\beta):= \mtHHH^{\dif_q}(\beta)^{\mathtt{f}_\beta}\cong R_\ell^{\mathtt{f}_\beta}\otimes_{R_\ell } \mtHHH^{\dif_q}(\beta).
\]
\item[(2)] Likewise, the  \emph{$H_q$-HOMFLYPT $p$-homology} is the doubly-graded $H_q$-module
\[
\pHHH^{\dif_q}(\beta,K+1):= \pHHH^{\dif_q}(\beta,K+1)^{\mathtt{f}_\beta}\cong R_\ell^{\mathtt{f}_\beta}\otimes_{R_\ell } \ptHHH^{\dif_q}(\beta,K+1).
\]
\end{enumerate}
\end{defn}
   
 \begin{cor}\label{cor-HHH-twisted}
Given a braid $\beta$, both  $\mHHH^{\dif_q}(\beta)$ and $\pHHH^{\dif_q}(\beta)$ are link invariants that only depend on the closure of $\beta$ as a link in $\R^3$.
 \begin{enumerate}
     \item[(i)] The slash homologies of $\mHHH^{\dif_q}(\beta)$ and $\pHHH^{\dif_q}(\beta,K+1)$ are finite-dimensional.
     \item[(ii)] Furthermore, the Euler characteristic of $\mHHH^{\dif_q}(\beta)$ is equal to the HOMFLYPT polynomial of $\widehat{\beta}$ in the formal variables $q$ and $a$, while the Euler characteristic of $\pHHH^{\dif_q}(\beta,K+1)$ is equal to the $\mathfrak{sl}_{K+1}$-polynomial of $\widehat{\beta}$ in a formal $q$-variable. 
     \item[(iii)] The Euler characteristic of the slash homology of $\mHHH^{\dif_q}(\beta)$ is equal to the specialization of the HOMFLYPT polynomial of $\widehat{\beta}$ at a root of unity $q$, while the Euler characteristic of the slash homology of $\pHHH^{\dif_q}(\beta,K+1)$ is the equal to the specialization of the $\mathfrak{sl}_{K+1}$-polynomial of $\widehat{\beta}$ at a root of unity $q$.
 \end{enumerate}
 \end{cor}  
\begin{proof}
For the first statement, just notice that the twisting of the $p$-DG structure by the framing factor takes care of the Markov II move.

Next, the finite-dimensionality of the homology theories follows, by construction, from the fact that ${^{f_{i_1}}B_{i_1}^{g_{i_1}}}\otimes_R \cdots \otimes_R {^{f_{i_m}}B_{i_m}^{g_{i_m}}}$ is an $H_q$-module with $2^m$-step filtration whose subquotients are isomorphic to $R^f$ as left $R\# H_q$-modules, and thus Corollary \ref{cor-finite-slash-homology} applies.
\end{proof}


\begin{rem}
The previous discussion in Section \ref{secmarkov2} forces us to make a specialization $a=q^{r}t$ in the homotopy category of
$t$ and $q$-bigraded $p$-complexes to obtain a framed Markov II invariance. In particular, when $r=K=0$, this forces the relation, on the Grothendieck group level, that $a=t=-1$. This specialization leads to a categorification of the Alexander skein relation.
\end{rem}

\section{Specialized homology theories} \label{secspecialhomology}

\subsection{A singly-graded homology}
Fix $k\in \N$. Consider the $H_q$-Koszul complex in one-variable: 
\begin{equation}\label{eqn-ordinary-Koszul-with-Cautis-d}
  C_1:  0 \lra a q^2 \Bbbk[x]^x\otimes \Bbbk[x]^{x} \stackrel{d_C}{\lra} \Bbbk[x]\otimes \Bbbk[x]\lra 0
\end{equation}
where $d_C$ is the map
$d_C(f)=(x^{kp+2}\otimes 1+1\otimes x^{kp+2})f$ and $k\in \N$.
We regard the differential on the arrow as an endomorphsim of the Koszul complex, of $(a,q)$-bidegree $(-1,2kp+2)$.

\begin{lem} \label{lemma-acylicity-commutator-relation}
The commutator of the endomorphisms $d_C$ and $\dif_q\in H_q$ is null-homotopic on the Koszul complex $C_1$.
\end{lem}
\begin{proof}
The commutator map $\phi:=[d_C,\dif_q]$ is given by
\[
\xymatrix{
 & 0 \ar[rr] && \Bbbk[x]^x\otimes \Bbbk[x]^{x} \ar[rr]^{x\otimes 1-1\otimes x} \ar[d]^{\phi}&& \Bbbk[x]\otimes \Bbbk[x] \ar[r] & 0 &\\
0\ar[r] &\Bbbk[x]^x\otimes \Bbbk[x]^{x} \ar[rr]^{x\otimes 1-1\otimes x} && \Bbbk[x]\otimes \Bbbk[x] \ar[rr] && 0  & 
}
\]
where $\phi$ maps the bimodule generator $1\otimes 1 \in \Bbbk[x]^x\otimes\Bbbk[x]^{x}$ as follows
\begin{align*}
    \phi(1\otimes 1) & = d_C(\dif_q(1\otimes 1)) -\dif_q d_C(1\otimes 1) =d_C(x\otimes 1+1\otimes x)-\dif_q(x^{kp+2}\otimes 1+1\otimes x^{kp+2})\\
     & = (x\otimes 1+1\otimes x)(x^{kp+2}\otimes 1+1\otimes x^{kp+2})-2(x^{kp+3}\otimes 1+1\otimes x^{kp+3}) \\
     & = -x^{kp+3}\otimes 1+x^{kp+2}\otimes x+x\otimes x^{kp+2}-1\otimes x^{kp+3} \\
     & = x^{kp+2}\otimes 1(1\otimes x-x\otimes 1)+(x\otimes 1-1\otimes x) 1\otimes x^{kp+2}\\
     & =  (x\otimes 1 - 1 \otimes x)(1\otimes x^{kp+2}-x^{kp+2}\otimes 1) .
\end{align*}
We may thus choose a null-homotopy to be
\[
\xymatrix{
 & 0 \ar[rr] && \Bbbk[x]^x\otimes \Bbbk[x]^{x}\ar[dll]_h \ar[rr]^{x\otimes 1-1\otimes x} \ar[d]^{\phi}&& \Bbbk[x]\otimes \Bbbk[x] \ar[r] & 0 &\\
0\ar[r] &\Bbbk[x]^x\otimes \Bbbk[x]^{x} \ar[rr]^{x\otimes 1-1\otimes x} && \Bbbk[x]\otimes \Bbbk[x] \ar[rr] && 0  & 
}
\]
where $h$ is given by multiplication by the element $1\otimes x^{kp+2}-x^{kp+2}\otimes 1$, and acts on the rest of the complex by zero. The result follows.
\end{proof}

The Koszul complex $C_n$ inherits the endomorphism $d_C$ by forming the $n$-fold tensor product from the one-variable case. It follows, that for a given $p$-DG bimodule $M$ over $R_n$, there is an induced differential, still denoted $d_C$, given via the identification
\begin{equation}
    \mHH_\bullet^{\dif_q}(M)\cong \mH_\bullet (M\otimes_{(R_n, R_n)} C_n),
\end{equation}
where the induced differential acts on the right hand side by $\Id_M \otimes d_C$. By construction, $d_C$ has Hochschild degree $-1$ and $q$-degree $2kp+2$.

Lemma \ref{lemma-acylicity-commutator-relation} immediately implies the following.

\begin{cor}\label{cor-dC-commutes-with-H}
The induced differential $d_C$ on $\mHH^{\dif_q}_\bullet(M)$ commutes with the $H_q$-action.\hfill$\square$
\end{cor}

\begin{rem}
The differential $d_C$, first observed by Cautis \cite{Cautisremarks}, has the following more algebro-geometric meaning. Identifying $\mHH^1(R_n)$ as vector fields on $\mathrm{Spec}(R_n)=\mathbb{A}^n$, $\mHH^1(R_n)$ acts as differential operators on $\mHH_\bullet(M)$ for any $(R_n,R_n)$-bimodule $M$, regarded as a coherent sheaf on $\mathbb{A}^n\times \mathbb{A}^n \cong T^* (\mathbb{A}^n)$. Under this identification, $d_C$ is given by, up to scaling by a nonzero number, contraction with the vector field
\[
\zeta_C:=\sum_{i=1}^n x_i^{kp+2}\frac{\dif}{\dif x_i}.
\]
On the other hand, $\dif_q$ is given by the polynomial derivation by the vector field
\[
\zeta_q:=\sum_{i=1}^n x_i^2 \frac{\dif}{ \dif x_i}.
\]
Since these two vector fields satisfy 
\[
\left[\zeta_C,\zeta_q\right] = \sum_{i,j} \left[x_i^{kp+2}\frac{\dif}{\dif x_i}, x_j^2 \frac{\dif}{\dif x_j}\right] =
\sum_{i}\left( 2x_i^{kp+3} \frac{\dif}{\dif x_i} - (kp+2) x_i^{kp+3} \frac{ \dif }{ \dif x_i} \right)= 0,
\]
the two actions naturally commute with each other on $\mHH_\bullet(M)$ via the Gerstenhaber module
structure on $\mHH_\bullet(M)$.

In a more general context, Hochschild homology is a Gerstenhaber module over Hochschild cohomology viewed
as a Gerstenhaber algebra.  We may view $d_C$ and $\dif_q$ as commuting elements in Hochschild
cohomology ring but the element $d_C$ acts on homology via cap product $\zeta_C \cap (\mbox{-}) $ and the element
$\dif_q$ acts via a Lie algebra action $\mathcal{L}_{\zeta_q}(\mbox{-})$.  The compatibility of these actions
is given by the equation
\[
\zeta_C \cap \mathcal{L}_{\zeta_q}(x)=[\zeta_C,\zeta_q] \cap x + \mathcal{L}_{\zeta_q}(\zeta_C \cap x).
\]
Since $[\zeta_C,\zeta_q]=0$, these actions commute.
\end{rem}

Now we are ready to introduce a further collapsed $p$-homology theory of a braid closure. Let $\beta\in \mathrm{Br}_n$ be an $n$-stranded braid. We have associated to $\beta$ a usual chain complex of $H_q$-equivariant Soergel bimodules $T_\beta$ as in \eqref{eqn-chain-complex-for-braid}, of which we take $\pHH^{\dif_q}_\bullet$ for each term: 
\begin{equation}\label{eqn-collapsed-diagram-1}
\begin{gathered}
   \xymatrix{ & \vdots & \vdots & \vdots & \\
 \dots \ar[r]^-{\dif_t} & \pHH_i^{\dif_q}(\pT_\beta^{m+1}) \ar[u]^{\dif_C} \ar[r]^-{\dif_t} & \pHH_i^{\dif_q}(\pT_\beta^{m})\ar[u]^{\dif_C}  \ar[r]^-{\dif_t} & \pHH_i^{\dif_q}(\pT_\beta^{m-1}) \ar[u]^{\dif_C} \ar[r]^-{\dif_t}  &\dots \\
 \dots \ar[r]^-{\dif_t} & \pHH_{i+1}^{\dif_q}(\pT_\beta^{m+1})  \ar[r]^{\dif_t} \ar[u]^{\dif_C} & \pHH_{i+1}^{\dif_q}(\pT_\beta^{m}) \ar[u]^{\dif_C} \ar[r]^-{\dif_t} & \pHH_{i+1}^{\dif_q}(\pT_\beta^{m-1}) \ar[u]^{\dif_C} \ar[r]^-{\dif_t} &\dots \\
              & \vdots \ar[u]^{\dif_C} & \vdots \ar[u]^-{\dif_C} & \vdots \ar[u]^{\dif_C} &
   } 
   \end{gathered}
\end{equation}
Here, $\dif_C$ is a $p$-differential arising from $d_C$ as follows. By \cite[Proposition 4.8]{QiSussanLink}, the $p$-Hochschild homology groups in a column above are identified with the terms in 
\begin{equation}\label{eqn-pH-construction}
\xymatrix@C=1.5em{
\cdots \ar[r]^-{d_C} & \mHH_{2i+1}^{\dif_q}(pT_\beta^m) \ar@{=}[r] & \cdots \ar@{=}[r] & \mHH_{2i+1}^{\dif_q}(pT_\beta^m) \ar[r]^-{d_C} & \mHH_{2i}^{\dif_q}(pT_\beta^m) \ar[r]^-{d_C} & \mHH_{2i-1}^{\dif_q}(pT_\beta^m) \ar@{=}[r] & \cdots 
},
\end{equation}
where each term in odd Hochschild degree is repeated $p-1$ times.  Here the horizontal differential is the 
$p$-Hochschild induced map of the topological differential, which we have denoted by $\dif_t$ to indicate its 
origin. On the arrows connecting even and odd Hochschild degree terms, we put the map $d_C$ while keeping the
repeated terms connected by identity maps. This defines a $p$-complex structure, denoted $\dif_C$,  in each column 
in diagram \eqref{eqn-collapsed-diagram-1}. The $p$-differential $\dif_C$ commutes with the $H_q$-action on each 
term by Corollary \ref{cor-dC-commutes-with-H}.  Denote the total $p$-differential $\dif_T:= \dif_t+\dif_C+\dif_q$, which collapses the double grading 
into a single $q$-grading. 

Let us also emphasize an important point about the vertical grading collapse as the following remark.

\begin{rem}
In order to $p$-extend the Koszul complex \eqref{eqn-ordinary-Koszul-with-Cautis-d} into a $p$-Koszul complex with $\dif_C$ of degree two, we are forced to make the functor specialization from $[1]^a_d=a$ into $q^{2kp+2}[1]^q_\dif$, so that the $p$-extended complex looks like
\begin{equation}\label{eqn-pC1-with-q-shift}
     pC_1:~  0 \lra  q^{2kp+4} \Bbbk[x]^x\otimes \Bbbk[x]^{x}[1]^q_\dif \stackrel{d_C}{\lra} \Bbbk[x]\otimes \Bbbk[x]\lra 0.
\end{equation}
Taking tensor products of $pC_1$, this determines the correct vertical $q$-degree shifts in each column of diagram \eqref{eqn-collapsed-diagram-1} of the $p$-Hochschild homology groups. 

Notice that, on the level of Grothendieck groups, this has the effect of specializing the formal variable $a$ into $-q^{2kp+2}$.
\end{rem}

When $[1]^t_{\dif}=[1]^q_{\dif}$ and $a=q^{2kp+2}[1]_{\dif}^q$, the braiding complexes \eqref{eqn-elementary-braids-general2} specialize to
\begin{equation}\label{eqn-elementary-braid-q-ext}
pT_i :=
 q^{-kp-3} \left(B_i  \xrightarrow{br_i} R[-1]^q_{\dif}\right)
,
\quad \quad \quad
pT_i' :=  q^{kp+3}\left(R[1]^q_\dif \xrightarrow{rb_i} q^{-2}  B_i^{-(x_i+x_{i+1})}\right)
.
\end{equation}
Comparing equations \eqref{eqn-elementary-braids-general2} with \eqref{eqn-elementary-braid-q-ext}, this forces
\begin{equation}
    K=kp+1.
\end{equation}
This also explains the necessity of $p$-extension in the collapsed $t$ and $a$ direction in $\pHHH$ in the previous section: the homological shift in that direction needs to be $p$-extended to agree with the homological shift in the $q$-direction.

Furthermore, the bigrading in diagram \eqref{eqn-collapsed-diagram-1} is now interpreted as a single grading, with both $\dif_C$ and $\dif_t$ raising $q$-degree by two.

\begin{defn} \label{pdgjonesdef} 
Let $\beta$ be an $n$ stranded braid. The \emph{untwisted $\mathfrak{sl}_{kp+2}$ 
$p$-homology} of $\beta$ is the slash homology group
\[
p\widehat{\mH}(\beta, kp+2):=
q^{-n(kp+1)}\mH^/_{\bullet}(\pHH^{\dif_q}_\bullet(pT_\beta),\dif_T),
\]
viewed as an object in $\mc{C}(\Bbbk,\dif_q)$. We will drop the $kp+2$ decoration whenever $k$ is fixed and clear from context.
\end{defn} 

The homology group $p\widehat{\mH}(\beta)$ is only singly-graded as an object in $\mc{C}(\Bbbk,\dif_q)$. By 
construction, $p\widehat{\mH}(\beta)$ is the slash homology with respect to the $\dif_T$ action on $\oplus_{i,j} 
\pHH_i^{\dif_q}(\pT_\beta^j)$ (see diagram \eqref{eqn-pH-construction}). The latter space is doubly-graded by the 
topological degree and $q$-degree with values in $\Z\times \Z$ (the Hochschild $a$ degree is already forced to be 
collapsed with the $q$ degree to make the Cautis differential $\dif_C$ homogeneous).  However, as in Section \ref{secmarkov2}, the Markov II invariance for the homology theory already requires one to collapse 
the $t$-grading onto the $a$-grading, thus also onto the $q$-grading.  We will use $p\widehat{\mH}_{i}(\beta)$ to 
stand for the homogeneous subspace sitting in some $q$-degree $i$.

This approach to a categorification of the Jones polynomial, at generic values of $q$, was first developed by Cautis \cite{Cautisremarks}.
We follow the exposition of Robert and Wagner from \cite{RW} and the closely related approach of Queffelec, Rose, and Sartori \cite{QRS}.

\subsection{Topological invariance}
In this subsection, we establish the topological invariance of the untwisted homology theory.

\begin{thm}\label{thm-untwisted-sl2-homology}
The homology $p\widehat{\mH}(\beta, kp+2)$ is a finite-dimensional framed link invariant depending only on the braid closure of $\beta$.
\end{thm}

\begin{proof}
The proof of the theorem will be similar to \cite[Theorem 5.6]{QiSussanLink}. It amounts to showing that taking slash homology of $\pHH_\bullet^{\dif}(\beta)$ with respect to $\dif_T$ satisfies the Markov II move.

We start by discussing the normal $H_q$-equivariant Hochschild homology version. Let $L$ be a link in $\R^3$ obtained as a braid closure $\widehat{\beta}$, where $\beta\in \mathrm{Br}_n$ is an $n$-stranded braid. Recall that the homology groups $\mHH^\dif_\bullet(L)$ are defined by tensoring a complex of Soergel bimodules $M$ determined by $\beta$ with the Koszul complex $C_n$ and computing its termwise vertical (Hochschild) homology.
The differential $d_C$ is defined on the Koszul complex $C_n$. To emphasize its dependence on $n$, we will write $d_C$ on $C_n$ as $d_n$ in this proof, and likewise write $\dif_n$ for the $p$-extended differential on $\pC_n$.

Since
\begin{equation*}
C_{n+1}=C_n\otimes C_1^\prime=C_n \otimes \Bbbk[x_{n+1}] \otimes \Lambda \langle dx_{n+1} \rangle \otimes \Bbbk[x_{n+1}],
\end{equation*}
the vertical differential may be inductively defined as
\begin{equation} \label{Drecursive}
d_{n+1}=d_n \otimes\mathrm{Id} +
\mathrm{Id}\otimes d_1^{\prime}.
\end{equation}
Here we have  set
$C_1^\prime=\Bbbk[x_{n+1}] \otimes \Lambda \langle dx_{n+1} \rangle \otimes \Bbbk[x_{n+1}] $ equipped with part of the Cautis differential
\[
{d}^\prime_{1}:=x_{n+1}^{kp+2} \otimes \iota_{\frac{\partial}{\partial x_{n+1}} }\otimes 1+ 1 \otimes \iota_{\frac{\partial}{\partial x_{n+1}} }\otimes x_{n+1}^{kp+2}
\ .
\]
The notation $\iota$ denotes the contraction of $dx_{n+1}$ with $\frac{\partial}{\partial x_{n+1}}$ .
Under $p$-extension,  write $\dif_C$ for the $p$-extended Cautis differential and  $\dif_1^\prime$ as the $p$-extended differential of $d_1^\prime$.

We start by re-examining diagram \eqref{sesbicomplexespHHH} with the shifts in \eqref{eqn-elementary-braid-q-ext}. It will, though, be helpful to keep the $a$ and $t$ gradings separate for the proof, with it understood that $[1]^a_\dif=q^{2kp+2}[1]^q_\dif$ and $[1]^t_{\dif}=[1]^q_\dif$. Thus we have a short exact sequence
\begin{equation} \label{eqnsesspecialized}
\begin{gathered}
\xymatrix@R=1em@C=1.4em{
& & & & & 0 \ar[dd] \\
  q^{-kp+1} R_{n+1}^{x_n+3x_{n+1}}  [1]_\dif^a \ar@/_4pc/[dddd]_{\phi} \ar[rrr]^-{2(x_{n+1}-x_n)} \ar[dd] & & &  q^{-kp-1} R_{n+1}^{2x_{n+1}} [1]_\dif^a  [-1]_\dif^t\ar[dd] \ar@/^2pc/[dddd]^{\Id} &  \\
& & & & := & pY_1 \ar[dddd]\\
0 \ar[rrr] & & & 0 \\ 
& & &  \\
q^{-kp-1} ({}^{x_{n+1}}B_n^{x_{n+1}}) [1]_\dif^a \ar[rrr]^{br} \ar[dd]^{x_{n+1} \otimes 1 - 1 \otimes x_{n+1}} \ar@/_4pc/[dddd]_{\tilde{br}} & & & q^{-kp-1} R_{n+1}^{2x_{n+1}} [1]_\dif^a [-1]_\dif^t \ar[dd]^{0} &  & \\
& & &  & = &
\pC_1^\prime \otimes_{(\Bbbk[x_{n+1}],\Bbbk[x_{n+1}])} \pT_n \ar[dddd] \\
  q^{-kp-3}B_n \ar[rrr]^{br} \ar@/_4pc/[dddd]_{2 \Id} & & & q^{-kp-3} R_{n+1} [-1]_\dif^t\ar@/^2pc/[dddd]^{2\Id} \\
&  &  & \\
q^{-kp-1} \tilde{R}_{n+1}^{x_n+x_{n+1}} [1]_\dif^a \ar[rrr] \ar[dd]^{(x_{n+1}-x_n) \otimes 1 - 1 \otimes (x_{n+1}-x_n)} & & & 0 \ar[dd] &  & \\
& & & & := & pY_2 \ar[dd] \\
q^{-kp-3}B_n  \ar[rrr]^{br} & & & q^{-kp-3}  R_{n+1}[-1]^t_\dif\\
& & & & & 0  
}
\end{gathered} \ .
\end{equation}
Further, the sequence splits as bimodules over $(R_n,R_n)$ (see the proof of \cite[Proposition 4.12]{QiSussanLink} for an explicit splitting).

We claim that, as modules over $\Bbbk[\dif_T]/(\dif_T^p)$, the $p$-homology groups $ \pHH_\bullet^{\dif_q}((M\otimes \Bbbk[x_{n+1}]) \otimes_{R_{n+1}} T_n) $ fit into a distinguished triangle
\begin{equation}\label{eqn-dc-filtration}
\mH^/_\bullet(pC_n \otimes_{(R_n,R_n)} (M \otimes_{R_n} pY_2))
\rightarrow
 \mHH_\bullet^{\dif_q}((M\otimes \Bbbk[x_{n+1}]) \otimes_{R_{n+1}} pT_n) 
\rightarrow
\mH^/_\bullet(pC_n \otimes_{(R_n,R_n)} (M \otimes_{R_n} pY_1)) 
\stackrel{[1]}{\rightarrow}
\end{equation}
after taking vertical slash  ($p$-Hochschild) homology. Note that this $p$-complex triangle is in reverse order of the above filtration \eqref{eqnsesspecialized}.

Indeed, since $\dif_C$ acts on the $pY_1$ and $pY_2$ tensor factors via $\dif_1^\prime$, it suffices to check that $\dif_1^\prime$ preserves the submodule arising from $pY_2$ and presents the part arising from $pY_1$ as a quotient. To do this, we re-examine the sequence \eqref{sesbicomplexespHHH} under vertical slash ($p$-Hochschild) homology, with the auxiliary $a$ and $t$-gradings. The part $pY_2$, under vertical homotopy equivalence, contributes to the horizontal (topological) complex
\begin{subequations}
\begin{equation}
    p{Y}^\prime_2:=\left(q^{-kp-3} R_{n+1} \xrightarrow{\dif_t=\mathrm{Id}} q^{-kp-3} R_{n+1}[-1]^t_\dif\right)
\end{equation}
sitting entirely in $p$-Hochschild degree $0$. Likewise, the part $pY_1$ contributes to the horizontal
\begin{equation}
   p{Y}^\prime_1:= \left(q^{-kp+1}R_{n+1}^{x_n+3x_{n+1}} [1]^a_\dif  \xrightarrow{\dif_t=2(x_{n+1}-x_n)}q^{-kp-1}R_{n+1}^{2x_{n+1}}[1]^a_\dif [-1]^t_\dif\right)
\end{equation}
\end{subequations}
sitting entirely in $p$-Hochschild degrees $1,\dots, p-1$.
Since $\dif_1^{\prime}$ decreases the $a$-degree by one (i.e.,~acting vertically downwards), $ p{Y}^\prime_2 $ must be preserved under $\dif_1^\prime$, acting upon it trivially, and $p{Y}^\prime_1 $ is equipped with the quotient action of $\dif_1^\prime$.

By the above discussion, $\dif_T=\dif_t+\dif_C+\dif_q$ acts on the term containing $pY_2^\prime$ only through $\dif_t+\dif_q$. Since this term is the cone of the identity map, it is null-homotopic and thus
\[
\mH^/_\bullet(pC_n\otimes_{(R_n,R_n)} (M\otimes_{R_n} pY_2) )\cong 0.
\]
Consequently, using that $[1]^a_\dif=q^{2kp+2}[1]^t_\dif$, we have an isomorphism
\begin{align*}
\mH_\bullet^/( \pHH_\bullet (( M\otimes \Bbbk[x_{n+1}])\otimes_{R_{n+1}} p T_n ),\dif_T)
& \cong \mH_\bullet^/(\pHH_\bullet( M \otimes_{R_n} pY_1^\prime),\dif_T) \\
& \cong q^{kp+1}\mH_\bullet^/(\pHH_\bullet(M),\dif_T)^{2x_n}.
\end{align*}
The $q^{kp+1}$ factor is cancelled out in the overall shift of $p\widehat{\mH}$. This finishes the first part of Markov II move.

The other case of the Markov II move is entirely similar, which we leave to the reader as an exercise.

Finally, the finite-dimensionality of $p\widehat{\mH}(\beta)$ follows from Corollary \ref{cor-finite-slash-homology}. The theorem follows.
\end{proof}

To obtain a categorical link invariant, we need to introduce a $p$-differential twisting to correct the framing factor occurring in Theorem \ref{thm-untwisted-sl2-homology}, as done in \cite[Section 5.3]{QiSussanLink}.
For a braid $\beta\in \mathrm{Br}_n $ whose closure is a framed link with $\ell$ components. Choose for each framed component of
$\widehat{\beta}$ in $\beta$ a single strand in $\beta$ that lies in that component after closure, say, the $i_r$th
strand is chosen for the $r$th component. Then define the polynomial ring $\Bbbk[x_{i_1},\dots, x_{i_\ell}]$ as a subring of $\Bbbk[x_1,\dots, x_n]$ generated by the chosen variables. Set
\begin{equation}
    \Bbbk[x_{i_1},\dots, x_{i_\ell }]^{\mathtt{f}_\beta}:= \Bbbk[x_{i_1},\dots, x_{i_\ell }]\cdot 1_\beta, \quad \quad \dif(1_\beta):=-\sum_{r=1}^\ell 2\mathtt{f}_r x_{i_r} 1_\beta.
\end{equation}
Then we make the twisting of $H_q$-modules on the $\pHH_\bullet$-level, termwise on $\pHH_\bullet(pT_\beta^i)$:
\begin{equation}
    \pHH^{\mathtt{f}_\beta}_\bullet(pT_\beta):=\pHH_\bullet(pT_\beta)\otimes_{\Bbbk[x_{i_1},\dots, x_{i_\ell}]}{\Bbbk[x_{i_1},\dots, x_{i_\ell}]}^{\mathtt{f}_\beta} .
\end{equation}

\begin{defn}
Given $\beta\in \mathrm{Br}_n$ whose closure is a framed link with $\ell$ components, the \emph{$\mathfrak{sl}_{kp+2}$ $p$-homology} is the object
\[
\pH(\beta,kp+2):= q^{-n(kp+1)}\mH^/_\bullet(\pHH^{\mathtt{f}_\beta}_\bullet(pT_\beta),\dif_T)
\]
in the homotopy category $\mc{C}(\Bbbk, \dif_q)$.
\end{defn}

As done for $\pHHH$, we will often drop $kp+2$ in the notation of the homology.

\begin{thm}
The $\mathfrak{sl}_{kp+2}$ $p$-homology $\pH(\beta,kp+2)$ is a singly-graded, finite-dimensional link invariant depending only on the braid closure of $\beta$ as a link in $\R^3$. Furthermore, when $k \in 2\Z$, its graded Euler characteristic 
$$\chi(\pH(L,kp+2)):=\sum_{i}q^i \mathrm{dim}_\Bbbk(\pH_{i}(L,kp+2))$$
is equal to the Jones polynomial evaluated at a $2p$th root of unity.
\end{thm}
\begin{proof}
The above framing twisting compensates for the
linear factors appearing in Markov II moves, thus establishing the topological 
invariance of $\pH(\beta)$. 
  
For the last statement, we will use the fact that the Euler characteristic does not change
before or after taking slash homology. This is because, as with the usual chain complexes, taking slash homology only gets rid of acyclic summands whose Euler characteristics are zero. 

Let us revisit diagram \eqref{eqn-pH-construction}. Before collapsing the $t$ and $q$-gradings, the diagram arises by $p$-extending  $\mHH_\bullet (T_\beta)$ in the vertical ($t$-)direction. Let $P_\beta(v,t)$ be the Poincar\'{e} polynomial of the bigraded complex  $\mHH_\bullet (T_\beta)$ where, for now, $v$, $t$ are treated as formal variables coming from $q$ and $t$ grading shifts.  As shown by Cautis \cite{Cautisremarks}, $P_\beta(v,-1)$ is the $\mathfrak{sl}_{kp+2}$ polynomial of the link $\widehat{\beta}$ in the variable $v$.

The $p$-extension in the topological direction is equivalent to categorically specializing $[1]^t_d$ to $[1]^q_\dif$. It has the effect, on the Euler 
characteristic level, of specializing $t=-1$. Thus we obtain that the Euler characteristic of $p\mH(\beta)$ is equal to $P_\beta(v=q,t=-1)$. This the evaluation of the $\mathfrak{sl}_{kp+2}$ polynomial evaluated at a $2p$th root of unity $q$. When $k \in 2\Z$, we have $q^{kp+2}=q^2$ in 
$$\mathbb{O}_p:=K_0(\mc{C}(\Bbbk,\dif_q))\cong \dfrac{\Z[q]}{(1+q^2+\cdots +q^{2(p-1)})},$$ 
so this evaluation is equal to the value of the Jones polynomial in $\mathbb{O}_p$.
\end{proof}

\section{Examples} \label{secexamples}
In this section we compute the various homologies constructed earlier for $(2,n)$ torus links $T_{2,n}$. Note that there are no framing factors to incorporate in this family of examples.  The calculations are straightforward modifications of the computations made in \cite{KRWitt} and adjusted for $p$-DG notions in \cite[Section 6]{QiSussanLink}.  We refer the reader to \cite{QiSussanLink} and just state the modified results here with minimal explanation.

Throughout the remainder of this subsection, let $R=\Bbbk[x_1,x_2]$,
$B=B_1$, and $T=T_1$.
\subsection{The HOMFLYPT homology for the \texorpdfstring{$(2,n)$}{(2,n)} torus link}



First note that the homology of the $n$-component unlink $L_0$ is 
\begin{equation*}
    \pHHH^{\dif_q}(L_0,K+1) \cong \bigotimes_{i=1}^n q^{-K} \left( \Bbbk[x_i] \oplus q^{2K+2} [1]^t_{\dif} \Bbbk[x_i]^{2x_i} \right)
    \ .
\end{equation*}

The following simplification of $T^{\otimes n}$ is proved in the same way as \cite[Lemma 6.1]{QiSussanLink}
\begin{lem} \label{complexforT^n}

In $\mc{C}^{\dif_q}(R,R,\dif_0)$, one has $T^{\otimes n} \cong (q^{-K-1}[-1]^t_{\dif})^n$
\begin{equation*}
\left(
\xymatrix{
q^{2(n-1)}B^{(n-1)e_1}[n]^t_{\dif} \ar[r]^-{p_n} 
& q^{2(n-2)}B^{(n-2)e_1}[n-1]^t_{\dif} \ar[r]^{\hspace{.5in}p_{n-1}}
& \cdots \ar[r]^{p_3 \hspace{.2in}}
& q^2 B^{e_1}[2]^t_{\dif} \ar[r]^{\hspace{.2in} p_{2}}
& B^{}[1]^t_{\dif} \ar[r]^{br} 
& R
}    
\right) \ ,
\end{equation*}
where
\begin{equation*}
p_{2i}=1 \otimes (x_2-x_1) - (x_2-x_1) \otimes 1
\quad \quad
p_{2i+1}=1 \otimes (x_2-x_1)+(x_2-x_1) \otimes 1.
\end{equation*}
\end{lem}

The following result is proved in the same way as \cite[Proposition 6.3]{QiSussanLink}
\begin{prop}
The bigraded $H_q$-HOMFLYPT $p$-homology of a $(2,n)$ torus knot, as an $H_q$-module depends on the parity of $n$.
\begin{enumerate}
\item[(i)] If $n$ is odd it is:
    \begin{align*}
    & q^{-nK-2n-2K}[-n]^t_{\dif} \left( q^{2K+2} [1]^t_{\dif} \Bbbk[x]^{2x} \oplus q^{4K+4} [2]^t_{\dif} \Bbbk[x]^{4x} \right)
    \bigoplus \\
    &\bigoplus_{i \in \{2,4,\ldots,n-1 \}} q^{-nK-2n-2K}[i-n]^t_{\dif} \left(
    q^{2(i-1)}\Bbbk[x]^{2(i-1)x}
    \oplus q^{2K}[1]^t_{\dif} \begin{pmatrix} q^{2i}\Bbbk[x] \\ \oplus  \\ q^{2i+2}\Bbbk[x] \end{pmatrix} 
    \oplus q^{2i+4+4K}[2]^t_{\dif} \Bbbk[x]^{2(i+1)x} \right)
    \end{align*}
    with the $H_q$-structure on the middle object 
    $\begin{pmatrix} \Bbbk[x] \\ \oplus  \\ \Bbbk[x] \end{pmatrix} $ given by
   $
        \begin{pmatrix}
        2i x & 0 \\
        -2 & (2i+2)x
        \end{pmatrix}
    $.
    \item[(ii)] If $n$ is even it is:
    \begin{align*}
    &q^{-nK-2n-2K}[-n]^t_{\dif}  \left( q^{2K+2} [1]^t_{\dif} \Bbbk[x]^{2x} \oplus q^{4K+4} [2]^t_{\dif} \Bbbk[x]^{4x} \right)
    \bigoplus \\ 
   & \bigoplus_{i \in \{2,4,\ldots,n-2 \}} q^{-nK-2n-2K}[i-n]^t_{\dif} 
    \begin{pmatrix}
    q^{2(i-1)}\Bbbk[x]^{2(i-1)x} \\
    \oplus \\ 
    q^{2K}[1]^t_{\dif} \begin{pmatrix} q^{2i}\Bbbk[x] \\ \oplus  \\ q^{2i+2}\Bbbk[x] \end{pmatrix}  \\
    \oplus \\ q^{2i+4+4K}[2]^t_{\dif} \Bbbk[x]^{2(i+1)x} \end{pmatrix}
    \\
   &  \bigoplus 
 q^{-nK-2n-2K}[-n]^t_{\dif} 
 \left( \begin{matrix}
    q^{2(n-1)}\Bbbk[x_1,x_2]^{(n-1)(x_1+x_2)} \\
    \oplus \\
    q^{2K}[1]^t_{\dif} \begin{pmatrix} q^{2n}\Bbbk[x_1,x_2] \\ \oplus  \\ q^{2n+2}\Bbbk[x_1,x_2] \end{pmatrix} \\
    \\ \oplus \\
    q^{2n+4+4K}[2]^t_{\dif} \Bbbk[x_1,x_2]^{(n+2)(x_1+x_2)}
    \end{matrix}
    \right)[n]^t_{\dif}
    \end{align*}
    with the $H_q$-structure on the middle object 
    $\begin{pmatrix} q^{2i}\Bbbk[x] \\ \oplus  \\ q^{2i+2}\Bbbk[x] \end{pmatrix} $ given by
   $
        \begin{pmatrix}
        2i x & 0 \\
        -2 & (2i+2)x
        \end{pmatrix}
    $
        and the $H_q$-structure on the middle object 
    $\begin{pmatrix} q^{2n}\Bbbk[x_1,x_2] \\ \oplus  \\ q^{2n+2}\Bbbk[x_1,x_2] \end{pmatrix} $ given by
   $
        \begin{pmatrix}
        (n+1)x_1+(n-1)x_2 & 0 \\
        -2 & n(x_1+x_2)+2x_2
        \end{pmatrix}
    $.
\end{enumerate}
\end{prop}

\begin{cor}
In the stable category of $H_q$-modules, the slash homology of the $H_q$-HOMFLYPT $p$-homology of a $(2,n)$ torus link $\pHHH^{\dif_q}(T_{2,n},K+1) $, depends on the parity of $n$.
\begin{enumerate}
\item[(i)] If $n$ is odd it is:
    \begin{align*}
    & q^{-nK-2n-2K}[-n]^t_{\dif} \left( q^{p+2K} V_{p-2}^q [1]^t_{\dif} \oplus q^{p+4K} V_{p-4}^q [2]^t_{\dif} \right)
    \bigoplus \\
    \bigoplus_{i \in \{2,4,\ldots,n-1 \}} & q^{-nK-2n-2K}[-n]^t_{\dif} \left(
  q^{p} V_{p-2(i-1)}^q 
    \oplus  \begin{pmatrix} q^{p+2K} V_{p-2i}^q   \\ \oplus  \\ q^{p+2K} V_{p-2i-2}^q  \end{pmatrix} [1]^t_{\dif}
    \oplus q^{p+2+4K} V_{p-2(i+1)}^q  [2]^t_{\dif} \right) [i]^t_{\dif}
    \ .
    \end{align*}
\item[(ii)] If $n$ is even it is:
    \begin{align*}
    & q^{-nK-2n-2K}[-n]^t_{\dif} \left( q^{p+2K} V_{p-2}^q  [1]^t_{\dif} \oplus q^{p+4K} V_{p-4}^q [2]^t_{\dif} \right)
    \bigoplus \\ 
    &\bigoplus_{i \in \{2,4,\ldots,n-2 \}} q^{-nK-2n-2K}[-n]^t_{\dif} \left(
    q^{p} V_{p-2(i-1)}^q 
    \oplus  \begin{pmatrix} q^{p+2K} V_{p-2i}^q  \\ \oplus  \\ 
    q^{p+2K} V_{p-2(i+1)}^q \end{pmatrix} [1]^t_{\dif}
    \oplus q^{p+2+4K} V_{p-2(i+1)}^q [2]^t_{\dif} \right) [i]^t_{\dif} \\
     &\bigoplus 
q^{-nK-2n-2K}[-n]^t_{\dif} \left(\begin{matrix}
 q^{2p} V_{p-(n-1)}^q \otimes  V_{p-(n-1)}^q  \\    
    \oplus  \\
    \left( q^{2p+2K} V_{p-n-1}^q \otimes V_{p-n+1}^q   \oplus 
    q^{2p+2K} V_{p-n}^q \otimes V_{p-n-2}^q  \right)
     [1]^t_{\dif} \\
    \oplus \\
    q^{2p+4K}V_{p-(n+2)}^q \otimes V_{p-(n+2)}^q [2]^t_{\dif}
    \end{matrix}
    \right) [n]^t_{\dif}
    \ .
    \end{align*}
\end{enumerate}
\end{cor}

\subsection{The \texorpdfstring{$\mf{sl}_{kp+2}$}{sl(kp+2)}-homology for the \texorpdfstring{$(2,n)$}{(2,n)} torus link}
To compute this homology, we will use the following tool. If $M_\bullet$
\[ 
M_\bullet = \left(
\cdots \xrightarrow{\dif_t} M_{i+1} \xrightarrow{ \dif_t } M_i \xrightarrow{ \dif_t} M_{i-1} \xrightarrow{\dif_t} \cdots
\right)
\]
we write $\mc{T}(M_\bullet)$ to be the total complex whose $p$-differential is the sum $\dif_T:=\dif_t+\dif_q$.

\begin{prop} \label{filtertrickprop} \cite[Proposition 6.6]{QiSussanLink}
Let $M_\bullet$ be a contractible $p$-complex of $H_q=\Bbbk[\dif_q]/(\dif_q^p)$-modules. Then the complex $(\mc{T}(M_\bullet),\dif_T=\dif_t+\dif_q)$ is acyclic.
\end{prop}

We will be applying Proposition \ref{filtertrickprop} in the following situation. Suppose $N_\bullet$ is a $p$-complex of $H_q$-modules whose boundary maps preserve the $H_q$-module structure. Further, let $M_\bullet$ be a sub $p$-complex that is closed under the $H_q$-action, and there is a map $\sigma$ on $M_\bullet$ as in Proposition \ref{filtertrickprop} that preserves the $H_q$-module structure. Then, when totalizing the $p$-complexes, we have $\mc{T}(M_\bullet) \subset \mc{T}(N_\bullet)$ and the natural projection map
\[
\mc{T}(N_\bullet) \lra \mc{T}(N_\bullet)/\mc{T}(M_\bullet)
\]
is a quasi-isomorphism. Similarly, if $M_\bullet$ is instead a quotient complex of $N_\bullet$ that satisfies the condition of Proposition \ref{filtertrickprop}, and $K_\bullet$ is the kernel of the natural projection map 
$$ 0 \lra K_\bullet \lra N_\bullet \lra M_\bullet \lra 0, $$
then the inclusion map of totalized complexes $\mc{T}(K_\bullet) \lra \mc{T}(N_\bullet)$ is a quasi-isomorphism.

We modify the the calculation in the previous section of the $(2,n)$ torus link to include the Cautis $p$-differential $\dif_C$.
Recall that in this singly-graded theory that 
$a=q^{2kp+2}[1]^q_\dif $ and $t=[1]^q_{\dif}$.

The Hochschild homology $\pHH_{\bullet}^{\dif_q}(R) $ with the induced Cautis differential $\dif_C$ is given by
\begin{equation} \label{pHHRwithdiffs}
\begin{gathered}
\xymatrix{
& R & \\
q^4 R^{2x_1}[1]^a_\dif  \ar[ur]^{x_1^{kp+2}} & & q^4 R^{2x_2} [1]^a_\dif \ar[ul]_{x_2^{kp+2}} \\
& q^8 R^{2e_1} [2]^a_\dif \ar[ul]_{-x_2^{kp+2}} \ar[ur]^{x_1^{kp+2}} &
}    \ .
\end{gathered}
\end{equation}
First we study $\pHH_\bullet^{\dif_q}(br) \colon \pHH_\bullet^{\dif_q} (B)[1]^q_\dif \rightarrow \pHH_\bullet^{\dif_q} (R)$. 
\begin{equation} \label{pHHBRwithdiffs}
\begin{gathered}
\xymatrix{
& R[1]^q_\dif  \ar@/^/[rrr]^{1 \mapsto 1} & & & R & \\
& q^{2kp+4} R[2]^q_\dif 
\ar@/^/[u]^{x_1^{kp+2} +x_2^{kp+2}}  & q^{2kp+6} R[2]^q_\dif 
\ar[ul]_{x_2^{kp+2}(x_2-x_1)} 
\ar@/^/[r]^{\begin{pmatrix}
1 & 0 \\
1 & x_2-x_1
\end{pmatrix}}
& q^{2kp+4} R^{2x_1}[1]^q_\dif  \ar[ur]^{x_1^{kp+2}}&  q^{2kp+4} R^{2x_2}[1]^q_\dif  \ar@/_/[u]_{x_2^{kp+2}}  \\
& q^{4kp+10} R^{3e_1}[3]^q_\dif 
\ar@/^/[u]^{x_2^{kp+2}(x_1-x_2)} \ar[ur]_{x_1^{kp+2} + x_2^{kp+2}} \ar@/_/[rrr]^{1 \mapsto (x_2-x_1)} & & & q^{4kp+8} R^{2e_1}[2]^q_\dif  \ar[ul]_{-x_2^{kp+2}} \ar@/_/[u]_{x_1^{kp+2}}
}    
\end{gathered}
\end{equation}
where the object $q^{2kp+4} R[2]^q_\dif \oplus q^{2kp+6} R[2]^q_\dif  $ in the left square is twisted by the matrix 
\begin{equation}
\begin{pmatrix} \label{connmatrix}
2x_1 & 0 \\
2 & x_1+3x_2
\end{pmatrix}
\ .
\end{equation}
Filtering the total complex \eqref{pHHBRwithdiffs} we obtain that it is quasi-isomorphic to
\begin{equation*}
\xymatrix{
\Bbbk \langle x_1^a x_2^b | 0 \leq a \leq kp+2, 0 \leq b \leq kp+1 \rangle [1]^q_\dif
\ar[r]^{\hspace{.3in} 1}
&
\Bbbk \langle x_1^a x_2^b | 0 \leq a,b \leq kp+1 \rangle
}    
\end{equation*}
which is quasi-isomorphic to
\begin{equation*}
\Bbbk \langle x_1^{kp+2}, x_1^{kp+2} x_2, \ldots, x_1^{kp+2} x_2^{kp+1} \rangle [1]^q_\dif
\end{equation*}
where
\begin{equation*}
\dif(x_1^{kp+2} x_2^j) = (kp+2+j) x_1^{kp+2} x_2^{j+1}
\ .
\end{equation*}
This is quasi-isomorphic to $q^5 V_1 [1]^q_\dif$ if $k=0$.
If $k>0$, it's quasi-isomorphic to 
\begin{equation*}
(q^{3p+2} V_{p-2} \oplus q^{4kp+4} V_2)[1]^q_{\dif}.
\end{equation*}

Next we analyze $\pHH_\bullet^{\dif_q}(p_{2i+1}) \colon \pHH_\bullet^{\dif_q}(q^{4i} B^{2ie_1}[2i+1]^q_\dif) {\lra} \mHH_\bullet^{\dif_q}(q^{4i-2}B^{(2i-1)e_1}[2i]^q_\dif)$
\begin{align} \label{pHHBBwithdiffs1}
& \pHH_\bullet^{\dif_q} (q^{4i} B^{2ie_1}[2i+1]^q_\dif)~=~ \nonumber \\
& \begin{gathered}
\xymatrix@C=0.75em{
& q^{4i} R^{2ie_1}[2i+1]^q_\dif   &   \\
q^{4i+2kp+4} R[2i+1]^q_\dif [1]^q_\dif  \ar[ur]^{x_1^{kp+2} +x_2^{kp+2}} & & q^{4i+2kp+6} R[2i+1]^q_\dif [1]^q_\dif \ar[ul]_{x_2^{kp+2}(x_2-x_1)} 
 \\
& q^{4i+4kp+10} R^{(2i+3)e_1}[2i+1]^q_\dif [2]^q_\dif  \ar[ul]^{x_2^{kp+2}(x_1-x_2)} \ar[ur]_{x_1^{kp+2} + x_2^{kp+2}}  & 
}    
\end{gathered} \ ,
\end{align} 
and
\begin{align} \label{pHHBBwithdiffs2}
&\pHH_\bullet^{\dif_q} (q^{4i-2} B^{(2i-1)e_1}[2i]^q_\dif)~=~ \nonumber \\
&\begin{gathered}
\xymatrix@C=1.5em{
& q^{4i-2}R^{(2i-1)e_1}[2i]^q_\dif & \\
q^{4i+2kp+2} R^{}[2i]^q_\dif [1]^q_\dif  \ar[ur]^{x_1^{kp+2}+x_2^{kp+2}}& & q^{4i+2kp+4} R^{}[2i]^q_\dif[1]^q_\dif  \ar[ul]_{x_2^{kp+2}(x_2-x_1)}  \\
& q^{4i+4kp+8} R^{(2i+2)e_1}[2i]^q_\dif[2]^q_\dif  \ar[ul]^{x_2^{kp+2}(x_2-x_1)} \ar[ur]_{x_1^{kp+2}+x_2^{kp+2}} &
}      
\end{gathered} \ ,
\end{align}
where the differentials for both objects in the middle horizontal rows of \eqref{pHHBBwithdiffs1} and \eqref{pHHBBwithdiffs2} are twisted by \eqref{connmatrix} and
$\pHH_\bullet^{\dif_q}(p_{2i+1})=2(x_2-x_1)$ (diagonal multiplication by $2(x_2-x_1)$).
Filtering this total complex yields the total complex
\begin{equation}
\begin{gathered}
\xymatrix{
q^{4i} \Bbbk \langle x_1^a x_2^b | 0 \leq a \leq kp+2, 0 \leq b \leq kp+1 \rangle [2i+1]^q_\dif
\ar[d]^{2(x_2-x_1)} \\
q^{4i-2} \Bbbk \langle x_1^a x_2^b | 0 \leq a \leq kp+2, 0 \leq b \leq kp+1 \rangle [2i]^q_\dif
}    
\end{gathered}\ .
\end{equation}
This is quasi-isomorphic to 
\begin{equation}
q^{4i} \Bbbk \langle x_1^{kp+2}, x_1^{kp+2} x_2, \ldots, x_1^{kp+2} x_2^{kp+1}  \rangle [2i+1]^q_\dif
\bigoplus 
q^{4i-2} \Bbbk \langle 1, x_1, \ldots, x_1^{kp+1} \rangle [2i]^q_\dif
\ 
\end{equation}
where the differential on the basis elements is given by 
\[
\begin{gathered}
\xymatrix{
x_1^{kp+2} \ar[d]_{kp+4i+2}& & 1 \ar[d]^{4i-2} \\
x_1^{kp+2}x_2 \ar[d]_{kp+4i+3} & \bigoplus & x_1 \ar[d]^{4i-1}  \\
\vdots \ar[d]_{kp+4i+kp+2} & & \vdots \ar[d]^{4i-2+kp}\\
x_1^{kp+2} x_2^{kp+1} & & x_1^{kp+1}
}
\end{gathered}
\ .
\]
Thus the total homology of this complex is isomorphic to 
$X_i :=$
\begin{equation} \label{defXi}
\begin{cases}
\left(\begin{matrix}
q^{4i} (q^{2(kp+2)+j}V_j \oplus q^{2(kp+2+j+1+(k-1)p)+p-j} V_{p-j})[2i+1]^q_\dif \\ \oplus \\
q^{4i-2} (q^{\bar{j}}V_{\bar{j}} \oplus q^{2((k-1)p+\bar{j}+1)+p-\bar{j}} V_{p-\bar{j}} )[2i]^q_\dif 
\end{matrix}
\right)
&
\text{ if } j, \bar{j} \neq 0 \\
\left(\begin{matrix}
q^{4i}(q^{2(kp+2)}V_0 \oplus q^{2(kp+2+kp+1)}V_0)[2i+1]^q_\dif \\
\oplus \\
q^{4i-2}(q^{\bar{j}}V_{\bar{j}} \oplus q^{2((k-1)p+\bar{j}+1)+p-\bar{j}} V_{p-\bar{j}} )[2i]^q_\dif
\end{matrix} \right)
& 
\text{ if } j=0, \bar{j} \neq 0 \\
\left(\begin{matrix}
q^{4i} (q^{2(kp+2)+j}V_j \oplus q^{2(kp+2+j+1+(k-1)p)+p-j} V_{p-j})[2i+1]^q_\dif \\
\oplus \\ q^{4i-2}(V_0 \oplus q^{2(kp+1)} V_0) [2i]^q_\dif
\end{matrix}
\right)
& 
\text{ if } j \neq 0, \bar{j} =0 \\
\end{cases}   \ 
\end{equation}
where $j \in \{0, \ldots, p \}$ such that $4i+2+j$ is divisible by $p$ and
$\bar{j} \in \{0, \ldots, p \}$ such that $4i-2+\bar{j}$ is divisible by $p$.

Once again when $n$ is even, the leftmost term in $T^{\otimes n}$ maps by zero into the rest of the complex so we have to understand the total homology of $\pHH_{\bullet}(q^{2(n-1)}B^{(n-1)e_1}[n]^q_\dif)$.
Filtering
\begin{equation} \label{HHBalonewithdiffs}
\begin{gathered}
\xymatrix{
& q^{2(n-1)}R^{(n-1)e_1}[n]^q_\dif  & \\
q^{2(n+1)+2kp} R^{(n-1)e_1}[n]^q_\dif [1]^q_\dif \ar[ur]^{x_1^{kp+2} +x_2^{kp+2}} & & q^{2(n+2)+2kp} R^{(n-1)e_1}[n]^q_\dif [1]^q_\dif  \ar[ul]_{x_2^{kp+2}(x_2-x_1)}   \\
& q^{2(n+4)+4kp} R^{(n+2)e_1}[n]^q_\dif [2]^q_\dif  \ar[ul]^{x_2^{kp+2}(x_1-x_2)} \ar[ur]_{x_1^{kp+2} + x_2^{kp+2}} &
}    
\end{gathered}
\end{equation}
where the middle terms
$q^{2(n+1)+2kp} R^{(n-1)e_1}[n]^q_\dif [1]^q_\dif \oplus q^{2(n+2)+2kp} R^{(n-1)e_1}[n]^q_\dif [1]^q_\dif$ are further twisted by the matrix \eqref{connmatrix},
yields that \eqref{HHBalonewithdiffs} is quasi-isomorphic to
\begin{equation} \label{defYn}
Y_{\frac{n}{2}}=
q^{2(n-1)}\Bbbk \langle x_1^a x_2^b | 0 
\leq a \leq kp+2, 0 \leq b \leq kp+1 \rangle [n]^q_\dif
\end{equation}
with a differential inherited from the polynomial algebra and twisted by $(n-1)e_1$.  
All of these computations together with an overall shift of 
$q^{-(n+2)kp-3n-2} [-n]^q_{\dif} $ yields the slash homology of the $(2,n)$ torus link for $k>0$.

\begin{equation} \label{pjonestorus}
\pH(T_{2,n},kp+2) \cong
\begin{cases}
q^{-(n+2)kp-3n-3n-2} [-n]^q_{\dif} \left(
(q^{3p+2} V_{p-2} \oplus q^{4kp+4} V_2)[1]^q_{\dif} \oplus \bigoplus_{i=1}^{\frac{n-1}{2}} X_i
\right) & \text{ if } 2 \nmid n \\
q^{-(n+2)kp-3n-3n-2} [-n]^q_{\dif} \left(
(q^{3p+2} V_{p-2} \oplus q^{4kp+4} V_2)[1]^q_{\dif} \oplus \bigoplus_{i=1}^{\frac{n-2}{2}} X_i
\oplus \mH^/_{\bullet} \left( Y_{\frac{n}{2}}\right)  \right) & \text{ if } 2 \mid n
\end{cases}
\end{equation}
where $X_i$ is the $p$-complex in \eqref{defXi} and
$Y_{\frac{n}{2}}$ is the $p$-complex in \eqref{defYn}.

Finding the homology of $Y_{\frac{n}{2}}$ is non-trivial.  In the example below we take $n=2$ which means we are computing part of the homology for the Hopf link.  We also take $k=1$ just for convenience of notation.

We thus need to compute the homology of $Z_1=$
\[
\xymatrix{
1 \ar[r]^{1} \ar[d]^{1} & x_2 \ar[r]^{2} \ar[d]^{1} & x_2^2 \ar[r]^{3} \ar[d]^{1} & \cdots \ar[r]^{p-1} & x_2^{p-1} \ar[r]^{0} \ar[d]^{1} & x_2^p \ar[r]^{1} \ar[d]^{1} & x_2^{p+1} \ar[r]^{-2} \ar[d]^{1} & \\
x_1 \ar[r]^{1} \ar[d]^{2} & x_1 x_2 \ar[r]^{2} \ar[d]^{2} & x_1x_2^2 \ar[r]^{3} \ar[d]^{2} & \cdots \ar[r]^{p-1} & x_1x_2^{p-1} \ar[r]^{0}  \ar[d]^{2} & x_1 x_2^p \ar[r]^{1} \ar[d]^{2} & x_1 x_2^{p+1} \ar[r]^-{-2} \ar[d]^{2} & \\
\vdots \ar[d]^{p-1} & \vdots \ar[d]^{p-1} & \vdots \ar[d]^{p-1} & \vdots  & \vdots \ar[d]^{p-1} & \vdots \ar[d]^{p-1} & \vdots \ar[d]^{p-1} & \\
x_1^{p-1} \ar[r]^{1} \ar[d]^{0} & x_1^{p-1} x_2 \ar[r]^{2} \ar[d]^{0} & x_1^{p-1} x_2^2 \ar[r]^{3} \ar[d]^{0} & \cdots \ar[r]^{p-1} & x_1^{p-1} x_2^{p-1} \ar[r]^{0}  \ar[d]^{0} & x_1^{p-1} x_2^p \ar[r]^{1} \ar[d]^{0} & x_1^{p-1} x_2^{p+1} \ar[r]^-{-2} \ar[d]^{0} & \\
x_1^{p} \ar[r]^{1} \ar[d]^{1} & x_1^{p} x_2 \ar[r]^{2} \ar[d]^{1} & x_1^{p} x_2^2 \ar[r]^{3} \ar[d]^{1} & \cdots \ar[r]^{p-1} & x_1^{p} x_2^{p-1} \ar[r]^{0}  \ar[d]^{1} & x_1^{p} x_2^p \ar[r]^{1} \ar[d]^{1} & x_1^{p} x_2^{p+1} \ar[r]^-{-2} \ar[d]^{1} & \\
x_1^{p+1} \ar[r]^{1} \ar[d]^{2} & x_1^{p+1} x_2 \ar[r]^{2} \ar[d]^{2} & x_1^{p+1} x_2^2 \ar[r]^{3} \ar[d]^{2} & \cdots \ar[r]^{p-1} & x_1^{p+1} x_2^{p-1} \ar[r]^{0}  \ar[d]^{2} & x_1^{p+1} x_2^p \ar[r]^{1} \ar[d]^{2} & x_1^{p+1} x_2^{p+1} \ar[r]^-{-2} \ar[d]^{2} & \\
x_1^{p+2} \ar[r]^{4}  & x_1^{p+2} x_2 \ar[r]^{5}  & x_1^{p+2} x_2^2 \ar[r]^{6}  & \cdots \ar[r]^{p+2} & x_1^{p+2} x_2^{p-1} \ar[r]^{p+3}   & x_1^{p+2} x_2^p \ar[r]^{p+4}  & x_1^{p+2} x_2^{p+1}   & \\
}
\]
where the arrows labeled $-2$ mean that the differential acts by
$x_1^j x_2^{p+1} \mapsto -2 x_1^{p+2}x_2^j$.

There is a large contractible summand $Z_2$ in the upper-left corner.  Then there is short exact sequence of complexes
\[
Z_2 \rightarrow Z_1 \rightarrow Z_3
\]
where $Z_3=$
\[
\xymatrix{
 & &  &  &  & x_2^p \ar[r]^{1} \ar[d]^{1} & x_2^{p+1} \ar[r]^{-2} \ar[d]^{1} & \\
 &  &  &  & & x_1 x_2^p \ar[r]^{1} \ar[d]^{2} & x_1 x_2^{p+1} \ar[r]^{-2} \ar[d]^{2} & \\
&  &  &   &  & \vdots \ar[d]^{p-1} & \vdots \ar[d]^{p-1} & \\
& &  &  &    & x_1^{p-1} x_2^p \ar[r]^{1} \ar[d]^{0} & x_1^{p-1} x_2^{p+1} \ar[r]^{-2} \ar[d]^{0} & \\
x_1^{p} \ar[r]^{1} \ar[d]^{1} & x_1^{p} x_2 \ar[r]^{2} \ar[d]^{1} & x_1^{p} x_2^2 \ar[r]^{3} \ar[d]^{1} & \cdots \ar[r]^{p-1} & x_1^{p} x_2^{p-1} \ar[r]^{0}  \ar[d]^{1} & x_1^{p} x_2^p \ar[r]^{1} \ar[d]^{1} & x_1^{p} x_2^{p+1} \ar[r]^{-2} \ar[d]^{1} & \\
x_1^{p+1} \ar[r]^{1} \ar[d]^{2} & x_1^{p+1} x_2 \ar[r]^{2} \ar[d]^{2} & x_1^{p+1} x_2^2 \ar[r]^{3} \ar[d]^{2} & \cdots \ar[r]^{p-1} & x_1^{p+1} x_2^{p-1} \ar[r]^{0}  \ar[d]^{2} & x_1^{p+1} x_2^p \ar[r]^{1} \ar[d]^{2} & x_1^{p+1} x_2^{p+1} \ar[r]^{-2} \ar[d]^{2} & \\
x_1^{p+2} \ar[r]^{4}  & x_1^{p+2} x_2 \ar[r]^{5}  & x_1^{p+2} x_2^2 \ar[r]^{6}  & \cdots \ar[r]^{p+2} & x_1^{p+2} x_2^{p-1} \ar[r]^{p+3}   & x_1^{p+2} x_2^p \ar[r]^{p+4}  & x_1^{p+2} x_2^{p+1}   & \\
} \ .
\]
The second row from the bottom with the rightmost column, along with the third row from the bottom and second column from the right give a contractible summand $Z_4$ of $Z_3$:
\[
Z_4= \Bbbk \langle x_1^{p+1}+x_2^{p+1}, \ldots, x_1^{p+1} x_2^{p-1} + x_1^{p-1} x_2^{p+1} \rangle \oplus 
\Bbbk \langle x_1^{p}+x_2^{p}, \ldots, x_1^{p} x_2^{p-1} + x_1^{p-1} x_2^{p} \rangle
\ .
\]
Then there is a short exact sequence of complexes
\[
Z_4 \rightarrow Z_3 \rightarrow Z_5
\]
where $Z_5=$
\[
\xymatrix{
x_1^{p} \ar[r]^{1} \ar[d]^{1} & x_1^{p} x_2 \ar[r]^{2} \ar[d]^{1} & x_1^{p} x_2^2 \ar[r]^{3} \ar[d]^{1} & \cdots \ar[r]^{p-1} & x_1^{p} x_2^{p-1} \ar[r]^{0}  \ar[d]^{1} & x_1^{p} x_2^p \ar[r]^{1} \ar[d]^{1} & x_1^{p} x_2^{p+1} \ar[r]^{-2} \ar[d]^{1} & \\
x_1^{p+1} \ar[r]^{1} \ar[d]^{2} & x_1^{p+1} x_2 \ar[r]^{2} \ar[d]^{2} & x_1^{p+1} x_2^2 \ar[r]^{3} \ar[d]^{2} & \cdots \ar[r]^{p-1} & x_1^{p+1} x_2^{p-1} \ar[r]^{0}  \ar[d]^{2} & x_1^{p+1} x_2^p \ar[r]^{1} \ar[d]^{2} & x_1^{p+1} x_2^{p+1} \ar[r]^{-2} \ar[d]^{2} & \\
x_1^{p+2} \ar[r]^{4}  & x_1^{p+2} x_2 \ar[r]^{5}  & x_1^{p+2} x_2^2 \ar[r]^{6}  & \cdots \ar[r]^{p+2} & x_1^{p+2} x_2^{p-1} \ar[r]^{p+3}   & x_1^{p+2} x_2^p \ar[r]^{p+4}  & x_1^{p+2} x_2^{p+1}   & \\
} \ .
\]
Now let $Z_6$ be the contractible subcomplex of $Z_5$ generated by $x_1^{p+1}$.
That is 
\[
Z_6=\Bbbk \langle x_1^{p+1}, 1! x_1^{p+1} x_2 + a_0 x_2^{p+1},
\ldots, (p-1)! x_1^{p+1} x_2^{p-1} + a_{p-2} x_1^{p+2} x_2^{p-2} \rangle 
\]
for some coefficients $a_0, \ldots, a_{p-2}$.
Then there is a short exact sequence of complexes
\[
Z_6 \rightarrow Z_5 \rightarrow Z_7
\]
where $Z_7=$
\[
\xymatrix{
x_1^{p} \ar[r]^{1}  & x_1^{p} x_2 \ar[r]^{2}  & x_1^{p} x_2^2 \ar[r]^{3}  & \cdots \ar[r]^{p-1} & x_1^{p} x_2^{p-1} \ar[r]^{0}   & x_1^{p} x_2^p \ar[r]^{1} \ar[d]^{1} & x_1^{p} x_2^{p+1} \ar[r]^{-2} \ar[d]^{1} & \\
 &  &  &  &    & x_1^{p+1} x_2^p \ar[r]^{1} \ar[d]^{2} & x_1^{p+1} x_2^{p+1} \ar[r]^{-2} \ar[d]^{2} & \\
x_1^{p+2} \ar[r]^{4}  & x_1^{p+2} x_2 \ar[r]^{5}  & x_1^{p+2} x_2^2 \ar[r]^{6}  & \cdots \ar[r]^{p+2} & x_1^{p+2} x_2^{p-1} \ar[r]^{p+3}   & x_1^{p+2} x_2^p \ar[r]^{p+4}  & x_1^{p+2} x_2^{p+1}   & \\
}
\]
Consider the contractible summand
\[
Z_8 = \Bbbk \langle x_1^p, \ldots , x_1^p x_2^{p-1} \rangle 
\ .
\]
Then there is a short exact sequence 
\[
Z_8 \rightarrow Z_7 \rightarrow Z_9
\]
where $Z_9=$
\[
\xymatrix{
 &  &   &  &   & x_1^{p} x_2^p \ar[r]^{1} \ar[d]^{1} & x_1^{p} x_2^{p+1} \ar[r]^{-2} \ar[d]^{1} & \\
 &  &  &  &    & x_1^{p+1} x_2^p \ar[r]^{1} \ar[d]^{2} & x_1^{p+1} x_2^{p+1} \ar[r]^{-2} \ar[d]^{2} & \\
x_1^{p+2} \ar[r]^{4}  & x_1^{p+2} x_2 \ar[r]^{5}  & x_1^{p+2} x_2^2 \ar[r]^{6}  & \cdots \ar[r]^{p+2} & x_1^{p+2} x_2^{p-1} \ar[r]^{p+3}   & x_1^{p+2} x_2^p \ar[r]^{p+4}  & x_1^{p+2} x_2^{p+1}   & \\
}
\]
We now easily decompose $Z_9$ into a sum of complexes 
\[
Z_9' \oplus Z_9'' \oplus Z_9''' \oplus Z_9''''
\]
where $Z_9'$ comes from the bottom row.
More specifically,
\[
Z_9' = \Bbbk \langle x_1^{p+2}, x_1^{p+2} x_2, \ldots, x_1^{p+2} x_2^{p-4}  \rangle 
\]
\[
Z_9'' = \Bbbk \langle   
x_1^{p+2} x_2^{p-3}, 
x_1^{p+2} x_2^{p-2},
x_1^{p+2} x_2^{p-1},
x_1^{p+2} x_2^{p},
x_1^{p+2} x_2^{p+1} 
\rangle
\]
\[
Z_9''' = \Bbbk \langle   
x_1^{p} x_2^{p} - \frac{1}{2} x_1^{p+2} x_2^{p-2}, 
x_1^{p+1} x_2^p + x_1^p x_2^{p+1} - x_1^{p+2} x_2^{p-1},
-x_1^{p+2} x_2^p + 2 x_1^{p+1} x_2^{p+1}
\rangle
\]
\[
Z_9''''=\Bbbk \langle 2 x_1^{p+2} x_2^{p-1} - 3 x_1^{p+1} x_2^p +3 x_1^p x_2^{p+1}  \rangle 
\]
Thus for $k=1$ and $n=2$ we get
\[
\mH^/_{\bullet} \left( Y_{\frac{2}{2}} \right) \cong
q^2[2]^q_{\dif}(q^{3p} V_{p-4} \oplus q^{4p+1}V_3 \oplus q^{4p+2 }V_2 \oplus q^{4p+2} V_0)
\ .
\]

\begin{rem}
If we repeat the above calculation for $k=2$ and $n=2$, everything would proceed in the same way.  Other than internal $q$-grading shifts, the homology 
$\mH^/_{\bullet} \left( Y_{\frac{2}{2}} \right)$ would be the same as above and contain objects $V_{p-4}, V_3, V_2, V_0$.

The homology of the Hopf link in \cite{QiSussanLink} does not contain objects of the form $V_{p-4}$ or $V_3$ (see \cite[Equation 6.17]{QiSussanLink}, in particular) in this tail part of the calculation.
Thus we obtain here a new categorification of the Jones polynomial at a $2p$th root of unity different from the original one constructed in \cite{QiSussanLink}.
\end{rem}

\addcontentsline{toc}{section}{References}


\bibliographystyle{alpha}
\bibliography{qy-bib}

%

\noindent Y.~Q.: { \sl \small Department of Mathematics, University of Virginia, Charlottesville, VA 22904, USA} \newline \noindent {\tt \small email: yq2dw@virginia.edu}

\vspace{0.1in}

\noindent J.~S.:  {\sl \small Department of Mathematics, CUNY Medgar Evers, Brooklyn, NY, 11225, USA}\newline \noindent {\tt \small email: jsussan@mec.cuny.edu \newline 
\sl \small Mathematics Program, The Graduate Center, CUNY, New York, NY, 10016, USA}\newline \noindent {\tt \small email: jsussan@gc.cuny.edu}

%
\end{document}